\documentclass[12pt, reqno]{amsart}
\usepackage{amssymb,verbatim,amscd,amsmath,graphicx}
\usepackage{graphicx}
\usepackage{caption}
\usepackage{subcaption}
\usepackage{pdfsync}
\usepackage{fullpage}
\usepackage{color}
\usepackage{enumitem}
\usepackage{tikz}
\usepackage{cleveref}

\usetikzlibrary{patterns}
\usetikzlibrary{calc}
\usetikzlibrary{matrix}
\usepgflibrary{arrows}
\usetikzlibrary{arrows}
\usetikzlibrary{decorations.pathreplacing}
\usetikzlibrary{decorations.pathmorphing}

\numberwithin{equation}{section}
\newtheorem{theorem}{Theorem}[section]
\newtheorem{lemma}[theorem]{Lemma}

\newtheorem{corollary}[theorem]{Corollary}

\theoremstyle{definition}
\newtheorem{definition}[theorem]{Definition}
\newtheorem{observation}[theorem]{Observation}
\newtheorem{def-prop}[theorem]{Definition-Proposition}

\newtheorem{example}[theorem]{Example}

\newtheorem{notation}[theorem]{Notation}

\newtheorem*{Mysketch}{Sketch of proof} 
  {\pushQED{\qed}\begin{Mysketch}}
  {\popQED\end{Mysketch}}

\DeclareMathOperator{\reg}{reg}

\DeclareMathOperator{\Ass}{Ass}

\DeclareMathOperator{\depth}{depth}

\DeclareMathOperator{\supp}{supp}

\newcommand{\ob}{\textbf{0}}
\newcommand{\GS}{\mathcal{G}}
\newcommand{\mm}{\mathfrak{m}}
\newcommand{\nn}{\mathfrak{n}}
\newcommand{\pp}{\mathfrak{p}}
\newcommand{\qq}{\mathfrak{q}}
\newcommand{\rr}{\mathfrak{r}}
\definecolor{violet}{rgb}{0.5,0,1}

\newcommand{\ZZ}{{\mathbb Z}}
\newcommand{\NN}{{\mathbb N}}

\def\1{{\bf 1}}
\def\0{{\bf 0}}

\newcommand{\Spec}{\text{Spec}}

\begin{document}

\title{Local Cohomology and Degree Complexes \\ of Monomial Ideals}

\author{Jonathan L. O'Rourke}
\address{Tulane University \\ Department of Mathematics \\
6823 St. Charles Ave. \\ New Orleans, LA 70118, USA}
\email{jorourk2@tulane.edu}
\urladdr{jonathanorourke.com}


\begin{abstract} This paper examines the dimension of the graded local cohomology $H_\mm^p(S/K^s)_\gamma$ and $H_\mm^p(S/K^{(s)})$ for a monomial ideal $K$. This information is encoded in the reduced homology of a simplicial complex called the degree complex. We explicitly compute the degree complexes of ordinary and symbolic powers of sums and fiber products of ideals, as well as the degree complex of the mixed product, in terms of the degree complexes of their components. We then use homological techniques to discuss the cohomology of their quotient rings. In particular, this technique allows for the explicit computation of $\reg ((I + J + \mm\nn)^{(s)})$ in terms of the regularities of $I^{(i)}$ and $J^{(j)}$.
\end{abstract}

\maketitle


\section{Introduction} \label{sec.intro}

	
	The local cohomology of the quotient ring $S/K$, with $K \subset S$ an ideal, encodes information about the computational complexity of $K$. Studying this cohomology has been a subject of increasing interest in recent years (see, e.g., \cite{DM17}, \cite{KPU15}, \cite{Rai18}). A popular program of research in recent years involves finding bounds for where local cohomology has nonzero dimension, via studying Castlenuovo-Mumford regularity and depth (see, e.g., \cite{AB17}, \cite{ABS19}, \cite{Ban15}, \cite{BBH18}, \cite{HNTT19}, \cite{JK19}, \cite{JNS18}, \cite{JS18}, \cite{MT19}, \cite{NSY17}). By calculating all of the local cohomology modules, we obtain a more complete picture of the complexity of the ideal.
	
	In particular, this paper examines what can be determined about the local cohomology modules of (ordinary and symbolic) powers of ideals that can be expressed as sums and fiber products of ideals in polynomial rings sharing no variables. The sum of ideals $I \subset A$ and $J \subset B$ in polynomial rings sharing no variables is of interest in an algebro-geometric sense, in that $\Spec(A/I) \times_k \Spec(B/J) = \Spec((A \otimes_k B)/(I + J))$ (see \cite{HNTT19}). For $I$ and $J$ edge ideals of hypergraphs, $I + J$ corresponds to the union of the two graphs. The \textit{fiber product} of $I \subset A$ and $J \subset B$, where $\mm$ and $\nn$ are the maximal homogeneous ideals of $A$ and $B$ respectively, is defined to be $I + J + \mm\nn$ (see \cite{NV19} and sources within). One may observe that ideals of this form can decompose the ring $(A \otimes_k B)/(I + J + \mm\nn) $ as a direct sum of rings: $A/I \oplus B/J$. Furthermore, if $I$ and $J$ are viewed as edge ideals of hypergraphs, $I + J + \mm\nn$ corresponds to the (graph-theoretic) \textit{join} of the hypergraphs. We are interested in studying powers of homogeneous ideals due to a celebrated result of Cutkosky, Herzog, and Trung (\cite{CHT99}), giving that, for $s \gg 0$, $\reg(I^s) = ds + b$, where $d$ depends on $I$. As yet, much less is known about the symbolic powers of these ideals.
	
	In this paper, we examine the dimension of $\ZZ^s$-graded local cohomology modules of $S/K^s$ and $S/K^{(s)}$, where $K$ is a sum or fiber product of monomial ideals. Due to a formula of Takayama \cite{Tak05}, this problem can be reduced to examining the reduced homology of a related simplicial complex called the \textit{degree complex}. The degree complex has been a useful tool of study for examining ideal invariants that are based on where the dimension of the local cohomology is nonzero, such as regularity and depth (see, e.g., \cite{HT16}, \cite{MT19}, \cite{TT14}). This paper seeks to compute these dimensions explicitly.
	
Let $A = k[x_1,\ldots,x_m]$, $B = k[x_{m+1},\ldots,x_n]$ be monomial ideals, $S = A \otimes_k B$, and let $\mm = (x_1,\ldots,x_m)$ and $\nn = (x_{m+1},\ldots,x_n)$ be the respective maximal homogeneous ideals of $A$ and $B$. Let $\Delta_X$ be the simplex on $\{1,\ldots,m\}$ and $\Delta_Y$ be the simplex on $\{m+1,\ldots,n\}$. The main results of the paper are how we may describe the degree complexes of (ordinary and symbolic) powers of sums and fiber products in terms of their constituent ideals.

\begin{theorem} Let $I \subset A$ and $J \subset B$ monomial ideals, let $\gamma = (\alpha,\beta) \in \ZZ^n$, and let $|\gamma| := \sum_{i=1}^n \gamma_i$. 
	\begin{enumerate}
	\item (\Cref{multiplying}) $\Delta_\gamma(IJ) = (\Delta_\alpha(I) \ast \Delta_Y) \cup (\Delta_X \ast \Delta_\beta(J))$
	\item (\Cref{powerofsum}) $\Delta_\gamma((I+J)^s) = \bigcup_{j=1}^s \Delta_\alpha(I^j) \ast \Delta_\beta(J^{s-j+1})$
	\item (\Cref{symbolicsum}) $\Delta_\gamma((I+J)^{(s)}) = \bigcup_{j=1}^s \Delta_\alpha(I^{(j)}) \ast \Delta_\beta(J^{(s-j+1)})$
	\item (\Cref{ordinaryfiberproduct}) Let $I$ and $J$ be additionally squarefree. Let 
	
$$\mathcal{A} = \begin{cases} \{\emptyset \subsetneq F \in \Delta_\alpha(I^{s-|\beta|})\}, & G_\beta = \emptyset, |\beta| < s \\ \emptyset, & G_\beta \neq \emptyset \text{ or } |\beta| \geq s \end{cases} \text{ be nonempty faces of } \Delta_\alpha(I^{s-|\beta|});$$ $$\mathcal{B} = \begin{cases} \{\emptyset \subsetneq F \in \Delta_\beta(J^{s-|\alpha|})\}, & G_\alpha = \emptyset, |\alpha| < s \\ \emptyset, & G_\alpha \neq \emptyset \text{ or } |\alpha| \geq s \end{cases} \text{ be nonempty faces of } \Delta_\beta(J^{s-|\alpha|}).$$ Then the nonempty faces of $\Delta_\gamma((I+J+\mm\nn)^s)$ are given by
	
	$$\{\emptyset \subsetneq F \in \Delta_\gamma((I+J+\mm\nn)^s)\} = \mathcal{A} \cup \mathcal{B}$$
	
	


	\item (\Cref{symbolicfiberproduct}) Let $I$ and $J$ be additionally squarefree. Let 
$$\mathcal{A'} = \begin{cases} \{\emptyset \subsetneq F \in \Delta_\alpha(I^{(s-|\beta|)})\}, & G_\beta = \emptyset, |\beta| < s \\ \emptyset, & G_\beta \neq \emptyset \text{ or } |\beta| \geq s \end{cases} \text{ be nonempty faces of } \Delta_\alpha(I^{(s-|\beta)|});$$ $$\mathcal{B'} = \begin{cases} \{\emptyset \subsetneq F \in \Delta_\beta(J^{(s-|\alpha|)})\}, & G_\alpha = \emptyset, |\alpha| < s \\ \emptyset, & G_\alpha \neq \emptyset \text{ or } |\alpha| \geq s \end{cases} \text{ be nonempty faces of } \Delta_\beta(J^{(s-|\alpha|)}).$$ Then the nonempty faces of $\Delta_\gamma((I+J+\mm\nn)^{(s)})$ are given by
	
	$$\{\emptyset \subsetneq F \in \Delta_\gamma((I+J+\mm\nn)^{(s)})\} = \mathcal{A'} \cup \mathcal{B'}$$

	\item (\Cref{gmp}) Let $I_1 \subseteq I_2 \subseteq A$ and $J_1 \subseteq J_2 \subseteq B$ be monomial ideals, and let $H = I_1 J_2 + I_2 J_1$ be the mixed product of $I_1, I_2, J_1, J_2$ (cf. \cite{HT10}). Let $\gamma = (\alpha,\beta) \in \ZZ^m \times \ZZ^{n-m}$. Then
$$\Delta_\gamma(H) = [\Delta_X \ast \Delta_\beta(J_2)] \cup [\Delta_\alpha(I_1) \ast \Delta_\beta(J_1)] \cup [\Delta_\alpha(I_2) \ast \Delta_Y]$$
	\end{enumerate}
\end{theorem}
	
From these decompositions, we may say the following about their graded local cohomologies as a consequence.
\begin{theorem} Let $I \subset A$ and $J \subset B$ be monomial ideals, let $\mm$ and $\nn$ be the respective maximal homogeneous ideals of $A$ and $B$, and let $\gamma = (\alpha,\beta) \in \ZZ^n$. 
	\begin{enumerate}
	\item (\Cref{cohomologyofsum})  $\dim_k H_{\mm+\nn}^p(S/(I+J))_\gamma = \sum_{u+v=p+1} \dim_k H_\mm^u(A/I)_\alpha \cdot \dim_k H_\nn^v(B/J)_\beta$
	\item (\Cref{cohomologyofproduct}) $\dim_k H_{\mm+\nn}^p(S/(IJ))_\gamma = \sum_{u + v = p} \dim_k H_\mm^u (A/I)_\alpha \cdot \dim_k H_\nn^v(B/J)_\beta$
	\item (\Cref{lescohom}) There is a collection of long exact sequences which relate $H_{\mm+\nn}^p(S/(I+J)^s)_\gamma$ to the modules $H_\mm^u(A/I^i)_\alpha$ and $H_\nn^v(B/J^j)_\beta$ and which relate $H_{\mm+\nn}^p(S/(I+J)^{(s)})_\gamma$ to the modules $H_\mm^u(A/I^{(i)})_\alpha$ and $H_\nn^v(B/J^{(j)})_\beta$ with $1 \leq i,j \leq s$.
	\item (\Cref{cohomologyfiberproduct}) 
	\begin{enumerate}
	\item Let $\diamond$ be the condition that each of $\Delta_\alpha(I^{s-|\beta|})$ and $\Delta_\beta(J^{s-|\alpha|})$ hsa a nonempty face and that $p = 1$. Then for $p \geq 1$, 
 	
	$\dim_k H_{\mm+\nn}^p(S/(I+J+\mm\nn)^s) = $
	$$\begin{cases} \dim_k H_\mm^p(A/I^{s-|\beta|})_\alpha + \dim_k H_\nn^p(B/J^{s-|\alpha|})_\beta, & \text{if not } \diamond \\
	\dim_k H_\mm^p(A/I^{s-|\beta|})_\alpha + \dim_k H_\nn^p(B/J^{s-|\alpha|})_\beta + 1, & \text{if } \diamond. \end{cases}$$
	\item Let $\diamond'$ be the condition that each of $\Delta_\alpha(I^{(s-|\beta|)})$ and $\Delta_\beta(J^{(s-|\alpha|)})$ has a nonempty face and that $p = 1$. Then	for $p \geq 1$,
	
	$\dim_k H_{\mm+\nn}^p(S/(I+J+\mm\nn)^{(s)}) = $
	$$\begin{cases} \dim_k H_\mm^p(A/I^{(s-|\beta|)})_\alpha + \dim_k H_\nn^p(B/J^{(s-|\alpha|)})_\beta, & \text{if not } \diamond' \\
	\dim_k H_\mm^p(A/I^{(s-|\beta|)})_\alpha + \dim_k H_\nn^p(B/J^{(s-|\alpha|)})_\beta + 1, & \text{if } \diamond'. \end{cases}$$	
	\end{enumerate}
	\end{enumerate}
\end{theorem}

As a further consequence of this description of the local cohomology, we have the following result about the regularity of symbolic powers of the fiber product.
	
\begin{theorem}
	(\Cref{depthandreg}) For $I \subset A$ and $J \subset B$ squarefree monomial ideals, $\mm$ and $\nn$ the respective maximal homogeneous ideals of $A$ and $B$, and $s \geq 1$, we have
$$\reg(S/(I+J+\mm\nn)^{(s)}) = \max_{1 \leq t \leq s} \{\reg(A/I^{(t)}) + s - t, \reg(B/I^{(t)}) + s - t, 2s - 1\}$$
\end{theorem}

	
	The paper is outlined as follows. In Section 2, we collect the necessary background and notations to be used throughout the paper. In Section 3, we describe the decomposition of the degree complexes of $IJ$, $I+J$, $(I+J)^s$ and $(I+J)^{(s)}$ and describe
	what we can derive about their graded local cohomologies. In Section 4, we describe the decomposition of the degree complexes of $(I+J+\mm\nn)^s$ and $(I + J + \mm\nn)^{(s)}$ and describe their graded local cohomologies. In Section 5, we describe the decomposition of $\Delta_\gamma(I_1 J_2 + I_2 J_1)$ for $I_1 \subseteq I_2 \subseteq A$ and $J_1 \subseteq J_2 \subseteq B$. Section 6 contains some useful functions for computing degree complexes in the Macaulay2 computer algebra system.
	
	\textbf{Acknowledgements.} The author wishes to thank his advisor, Huy T\`ai H\`a, for suggesting the strategy for determining local cohomology and for providing direction for the paper. The author also wishes to thank Selvi Kara Beyarslan for carefully reading through the paper, and her many helpful comments, suggestions, and discussions. The author would also like to thank Hop D. Nguyen for pointing out some errors in the first version of this paper.

\section{Preliminaries} \label{sec.prel}

In this section, we will introduce the necessary background and fix the notation to be used in the remainder of the paper.

Throughout this paper, let $k$ be a field of characteristic zero, and define sets of variables $X = \{x_1,\ldots,x_m\}$ and $Y = \{x_{m+1},\ldots,x_n\}$. Let $A = k[X] = k[x_1,\ldots,x_m]$, $B = k[Y] = k[x_{m+1},\ldots,x_n]$, and $S = A \otimes_k B = k[x_1,\ldots,x_n]$. For $\gamma \in \ZZ^n$, define $x^\gamma := x_1^{\gamma_1} x_2^{\gamma_2} \cdots x_n^{\gamma_n}$. 

\begin{definition} A \textit{simplicial complex} $\Delta$ is a collection of subsets $\sigma \subset [n]$, called \textit{faces}, that is closed under taking subsets; i.e., if $\sigma \in \Delta$ and $\sigma' \subseteq \sigma$, then $\sigma' \in \Delta$. The \textit{(simplicial) join} of two simplicial complexes $\Delta$ and $\Gamma$ is given by $$\Delta \ast \Gamma = \{\sigma \cup \tau : \sigma \in \Delta, \tau \in \Gamma\}.$$ A \textit{simplex} on a set $T \subset [n]$, denoted $\Delta_T$, has the full set $T$ as a face, and therefore $\tau \in \Delta_T$ for all $\tau \subset T$. The \textit{void complex} is the complex with no faces, and the \textit{irrelevant complex} is the complex whose only face is $\emptyset$. Observe that a simplicial complex joined with the void complex is again the void complex.  Note that $\tilde{H}_p(\emptyset;k) = 0$ for all $p$ and $\tilde{H}_p(\{\emptyset\};k) = 0$ for $p \neq -1$, but $\tilde{H}_p(\{\emptyset\};k) = k$ for $p = -1$.
\end{definition}

\begin{notation} Let $[n] := \{1,\ldots,n\}$. For $S = k[x_1,\ldots,x_n]$ and $F \subseteq [n]$, define $S_F$ to be $S[x_i^{-1} : i \in F]$. For $\gamma \in \ZZ^n$, denote by $G_\gamma := \{i \in [n]: \gamma_i < 0\}$.
\end{notation}

\begin{definition} \label{degcpx} For a monomial ideal $I \subset S$ and $\gamma \in \ZZ^n$, the \textit{degree complex} $\Delta_\gamma(I)$ is a simplicial complex defined as
$$\Delta_\gamma(I) = \{F \subseteq [n] \setminus G_\gamma : x^\gamma \not\in IS_{F \sqcup G_\gamma}\}.$$
\end{definition}

The reason for studying the degree complexes of a monomial ideal is a formula due to Takayama (originally \cite{Tak05}; the version used here can be seen in \cite{MT19}).

\begin{theorem} \label{takayama} Let $I \subset S$ be a monomial ideal.
$$\dim_k H_{\mathfrak{m}}^i(R/I)_\gamma = \begin{cases} \dim_k \tilde{H}_{i - |G_\gamma| - 1}(\Delta_\gamma(I);k), & \text{if } G_\gamma \in \Delta_\textbf{0}(I) \\ 0, & \text{otherwise.} \end{cases}$$
\end{theorem}

In particular, because depth and Castelnuovo-Mumford regularity can be defined based on when the cohomology modules $H_{\mathfrak{m}}^i(R/I)$ are nonzero, a full description of the degree complex can help determine these values.

\begin{definition} Let $M$ be a module with support $\mm = (x_1,\ldots,x_n)$. Then \begin{itemize}
\item $\text{depth}(M) = \min\{p : H_\mm^p(M) \neq 0\} = \min\{p: H_\mm^p(M)_\gamma \neq 0 \text{ for some } \gamma \in \ZZ^n\}$
\item $\text{reg}(M) = \max\{p + |\gamma| : p = 0,\ldots,\dim M, H_\mm^p(M)_\gamma \neq 0\}$.
\end{itemize}
\end{definition}

\begin{observation} \label{subsets} These are some useful facts about the degree complex.
\begin{enumerate}
\item If $I_1 \subseteq I_2$, then $\Delta_\gamma(I_1) \supseteq \Delta_\gamma(I_2)$.
\item For any $\gamma \in \ZZ^n$, $\Delta_\gamma(1) = \{\}$.
\item $F \in \Delta_\ob(I)$ if and only if for all monomials $m \in I, \supp(m) \not\subseteq F$
\item For any nontrivial monomial ideal $I \subset S$, $\Delta_\ob(I) = \Delta_\ob(\sqrt{I})$.
\end{enumerate}
\end{observation}
\begin{proof}
\begin{enumerate}
\item For $F \in [n] \setminus G_\gamma$,
\begin{equation} \notag
\begin{aligned}
F \in \Delta_\gamma(I_2) & \Longleftrightarrow & & x^\gamma \not\in I_2 S_{F \sqcup G_\gamma} \\
& \Longrightarrow & & x^\gamma \not\in I_1 S_{F \sqcup G_\gamma} \subseteq I_2 S_{F \sqcup G_\gamma} \\
& \Longleftrightarrow & & F \in \Delta_\gamma(I_1)
\end{aligned}
\end{equation}
\item For $F \subseteq [n] \setminus G_\gamma$, $F \in \Delta_\gamma(1)$ iff $x^\gamma \not\in S_{F \sqcup G_\gamma}$. However, $$x^\gamma = \prod_{i \not\in G_\gamma} x_i^{\gamma_i} \cdot \prod_{i \in G_\gamma} x_i^{\gamma_i} \in S \cdot S_{G_\gamma} \subseteq S_{F \sqcup G_\gamma}.$$
Therefore, for any $F \subseteq [n] \setminus G_\gamma$, $F \not\in \Delta_\gamma(1)$, so $\Delta_\gamma(1)$ is the void complex.
\item If $\supp(m) \subseteq F$ for some $m \in I$, then $m^{-1} \in S_F$, and $1 = m \cdot m^{-1} \in I S_F$, so that $F \in \Delta_\ob(I)$. If $\supp(m) \not\subseteq F$ for each $m \in I$, then because $1 \not\in I$ and $m^{-1} \not\in S_F$ for any $m \in I$, we have $1 \not\in IS_F$, so that $F \in \Delta_\ob(I)$.
\item For each $m \in I$, the squarefree monomial $m_{sqfree} := \prod_{i \in \supp(m)} x_i \in \sqrt{I}$, and $\supp(m) = \supp(m_{sqfree})$. We also have $1 \in IS_F$ if and only if, for some $0 \neq m \in I$, $\supp(m) \subseteq F$. Therefore, we have
\begin{equation} \notag
\begin{aligned}
F \in \Delta_\ob(I) & \Longleftrightarrow & & 1 \in IS_F \\
& \Longleftrightarrow & & \text{for some } 0 \neq m \in I, \hspace{1mm} \supp(m) \subseteq F \\
& \Longleftrightarrow & & \text{for some } 0 \neq m_{sqfree} \in \sqrt{I}, \hspace{1mm} \supp(m_{sqfree}) = \supp(m) \subseteq F \\
& \Longleftrightarrow & & 1 \in \sqrt{I}S_F \\
& \Longleftrightarrow & & F \in \Delta_\ob(\sqrt{I}).
\end{aligned}
\end{equation}
\end{enumerate}
\end{proof}

Once the decomposition of the degree complex is given, many of the results for ordinary powers immediately yield results for \textit{symbolic powers} via identical proofs.

\begin{definition} (\cite{CEHH17}, \cite{Vil01}) Let $I \subset S$ be a homogeneous ideal. The $s$th symbolic power of $I$, denoted by $I^{(s)}$, is given by $$I^{(s)} = S \cap \bigcap_{\pp \in \text{Ass}(I)} \pp^s S_\pp,$$ where $\text{Ass}(I)$ denotes the set of associated primes of $I$. If $I$ is a squarefree monomial ideal, this is simply $$I^{(s)} = \bigcap_{\pp \in \text{Ass}(I)} \pp^s.$$
\end{definition}


\section{Degree Complexes of Powers of Sums of Ideals} \label{sec.powersum}

In this section, we will describe how operations on ideals translate to operations on degree complexes, and we will provide a decomposition of the degree complexes of $I + J$ and $I \cap J$. Furthermore, when $I \subset A$ and $J \subset B$, we provide a decomposition of the degree complexes of $IJ$, $(I+J)^s$, and $(I+J)^{(s)}$.

\begin{notation} For a monomial ideal $I \subset S$, define $\supp(I) := \bigcup_{m \in \GS(I)} \supp(m)$, where $\GS(I)$ is the set of minimal generators of $I$. Let $k[F] := k[x_i : i \in F]$.
\end{notation}

First, we prove some lemmas relating to membership in the ideal $IS_F$ in order to simplify calculations when we are considering faces of the degree complex.

The following lemma indicates that we may change $\gamma_i$ for $i \in F$ without affecting whether the face $F$ is included.

\begin{lemma} \label{freeexponents} Let $I \subset S$ be a monomial ideal, $F \subseteq [n]$, $\gamma \in \ZZ^n$, and $x^\delta \in k[F]$. Then $x^\gamma \in IS_F$ if and only if $x^\delta x^\gamma \in IS_F$.
\end{lemma}
\begin{proof} Let $x^\delta \in k[F]$. Clearly, if $x^\gamma \in IS_F$, then $x^\delta x^\gamma \in IS_F$. \\
If $x^\delta x^\gamma \in IS_F$, then because $x^{-\delta} \in S_F$, we have $x^\gamma = x^\delta x^\gamma \cdot x^{-\delta} \in IS_F.$
\end{proof}

The following lemma reduces the question of whether $x^\gamma \in IS_F$ to whether $x^\delta x^\gamma \in I$ for some $x^\delta \in k[F]$.

\begin{lemma} \label{facelemma} For a monomial ideal $I \subset S$ and a subset $F \subseteq [n]$, the monomial  $x^\gamma \in IS_F$ if and only if there exists some monomial $x^\delta \in k[F]$ such that $x^\delta x^\gamma \in I$.
\end{lemma}

\begin{proof} If there exists $x^\delta \in k[F]$ such that $x^\delta x^\gamma \in I$, then $x^\gamma = x^\delta x^\gamma \cdot x^{-\delta} \in I \cdot S_F$. On the other hand, if $x^\gamma \in IS_F$, then we may write $x^\gamma = mx^{-\delta}$, where $m \in I$ is a monomial and $x^\delta \in k[F]$ (so that $x^{-\delta} \in S_F$). In this case, $x^\delta x^\gamma = x^\delta mx^{-\delta} = m \in I$.
\end{proof}

The following lemma allows us to reduce to the case where $\gamma \in \NN^n$ for the above lemma.

\begin{lemma} \label{removenegs} Let $I \subset S$ be a monomial ideal, $\gamma \in \ZZ^n$, and let $\gamma' \in \NN^n$ be defined by $$\gamma_i' := \begin{cases} \gamma_i, & \text{if } \gamma_i \geq 0 \\ 0, & \text{if } \gamma_i < 0 \end{cases}.$$ Then for any $F \subseteq [n] \setminus G_\gamma$, $x^\gamma \in IS_{F \sqcup G_\gamma}$ iff $x^{\gamma'} \in IS_{F \sqcup G_\gamma}$.
\end{lemma}
\begin{proof} 
\begin{equation} \notag
\begin{aligned}
x^{\gamma'} \in IS_{F \sqcup G_\gamma} & \Longleftrightarrow & & x^{\gamma'} = (\prod_{i \in G_\gamma} x_i^{-\gamma_i}) x^{\gamma} \in IS_{F \sqcup G_\gamma} \\
& \Longleftrightarrow & & x^\gamma \in IS_{F \sqcup G_\gamma} & \text{(Lemma \ref{freeexponents})}
\end{aligned}
\end{equation}
\end{proof}

The following lemma allows us to perform addition in the ring $S_F$.

\begin{lemma} \label{additionsplitlemma} Let $I$ and $J$ be monomial ideals in $S$, and let $\gamma \in \ZZ^n$. Then for any $F \subset [n]$, $x^\gamma \in (I + J)S_F$ if and only if $x^\gamma \in IS_F + JS_F$.
\end{lemma} 

\begin{proof} Let $I$ and $J$ be monomial ideals, $\gamma \in \ZZ^n$, and $F \subseteq [n]$.

\begin{equation} \notag
\begin{aligned}
x^\gamma \in (I+J)S_F & \Longleftrightarrow & & \text{there exists } x^\delta \in k[F] \text{ s.t. } x^\delta x^\gamma \in (I + J) \text{ (Lemma \ref{facelemma})} \\
& \Longleftrightarrow & & \text{there exists } x^\delta \in k[F] \text{ such that } x^\delta x^\gamma \in I \text{ or } x^\delta x^\gamma \in J \\
& \Longleftrightarrow & & x^\gamma \in IS_F \text { or } x^\gamma \in JS_F \\
& \Longleftrightarrow & & x^\gamma \in IS_F + JS_F
\end{aligned}
\end{equation}
\end{proof}

The following lemma is straightforward, but we include it for clarity. It allows us to reduce the question of whether $x^\gamma \in I$ to the question of whether $\prod_{i \in T} x_i^{\gamma_i} \in I$, for $T \supseteq \supp(I)$.

\begin{lemma} \label{reductionlemma1} Let $I \subset S$ be a monomial ideal, $\supp(I) \subseteq T \subseteq [n]$ and $\gamma \in \NN^n$. Then $x^\gamma \in I$ if and only if $\prod_{i \in T} x_i^{\gamma_i} \in I$.
\end{lemma}

\begin{proof} If $\prod_{i \in T} x_i^{\gamma_i} \in I$, then $x^\gamma = \prod_{i \in [n] \setminus T} x_i^{\gamma_i} \cdot \prod_{i \in T} x_i^{\gamma_i} \in I$. On the other hand, if $x^\gamma \in I$, then there exists some $m \in \GS(I)$ such that $m$ divides $x^\gamma$. Because $\supp(m) \subseteq \supp(I) \subseteq T$, $m$ must also divide $\prod_{i \in T} x_i^{\gamma_i}$, so that $\prod_{i \in T} x_i^{\gamma_i} \in I$ as well.
\end{proof}

With these proved, we may prove the following lemma that collects the variables not in any generator of $I$ as a simplex.

\begin{lemma} \label{reductionlemma2} Let $I \subset S$ be a monomial ideal, $\supp(I) \subseteq T \subseteq [n]$, and $\gamma \in \ZZ^n$, and let $\gamma_T = (\gamma_i : i \in T) \in \ZZ^{|T|}$. Then $\Delta_\gamma(I) = \Delta_{\gamma_T}(I) \ast \Delta_{[n] \setminus T}$, where $\Delta_{[n] \setminus T}$ is the simplex on $[n] \setminus T$.

\end{lemma}

\begin{proof} For any $F \subseteq [n] \setminus G_\gamma$, we have
\begin{equation} \notag
\begin{aligned}
F \in \Delta_\gamma(I) & \Longleftrightarrow & & x^\gamma \not\in IS_{F \sqcup G_{\gamma}} & & \text{(Definition \ref{degcpx})} \\
&  \Longleftrightarrow & & \text{there exists no } x^\delta \in k[F \sqcup G_\gamma] \text{ such that } x^\delta x^\gamma \in I & & \text{(Lemma \ref{facelemma})} \\
& \Longleftrightarrow & & \text{ for all } x^\delta \in k[F \sqcup G_\gamma], \text{ } \prod_{i \in T} x_i^{\delta_i} x_i^{\gamma_i} \not\in I & & \text{(Lemma \ref{reductionlemma1})} \\
& \Longleftrightarrow & & \prod_{i \in T} x_i^{\gamma_i} \not\in IS_{(F \sqcup G_\gamma) \cap T} \\
& \Longleftrightarrow & & x^{\gamma_T} = \prod_{i \in T} x_i^{\gamma_i} \not\in IS_{(F \sqcup G_{\gamma}) \cap T} = IS_{(F \cap T) \cup G_{\gamma_T}} \\
& \Longleftrightarrow & & F \cap T \in \Delta_{\gamma_T}(I)
\end{aligned}
\end{equation}

Therefore $F \in \Delta_\gamma(I)$ if and only if $F = (F \cap T) \cup (F \cap ([n] \setminus T))$ with $F \cap T \in \Delta_{\gamma_T}(I)$ and no restriction on $F \cap ([n] \setminus T)$. Hence, $\Delta_\gamma(I) = \Delta_{\gamma_T}(I) \ast \Delta_{[n] \setminus T}$.

\end{proof}

\begin{notation} \label{convention} Let $\gamma = (\alpha,\beta) \in \ZZ^m \times \ZZ^{n-m}$, so that $\gamma_i = \alpha_i$ for $i \in X = \{1,\ldots,m\}$ and $\gamma_i = \beta_{i-m}$ for $i \in Y = \{m+1,\ldots,n\}$. By convention, let $\Delta_\alpha(\cdot) \subseteq \Delta_X$ and $\Delta_\beta(\cdot) \subseteq \Delta_Y$. 
\end{notation}

With these lemmas and conventions in place, we can prove the decompositions of degree complexes of sums and intersections, which we do in the following three theorems.

\begin{theorem} \label{adding} Let $I$ and $J$ be monomial ideals in $S = k[x_1,\ldots,x_n]$, and let $\gamma \in \ZZ^n$.
\begin{enumerate}
\item $\Delta_\gamma(I + J) = \Delta_\gamma(I) \cap \Delta_\gamma(J)$.
\item If $I \subset A = k[x_1,\ldots,x_m]$, $J \subset B = k[x_{m+1},\ldots,x_n]$, and $\gamma = (\alpha,\beta) \in \ZZ^m \times \ZZ^{n-m}$, then $\Delta_\gamma(I + J) = \Delta_\alpha(I) \ast \Delta_\beta(J)$.
\end{enumerate}
\end{theorem}

\begin{proof} 
\begin{enumerate}
\item For any $F \subseteq [n] \setminus G_\gamma$, we have 
\begin{equation} \notag
\begin{aligned}
F \in \Delta_\gamma(I+J) & \Longleftrightarrow & & x^\gamma \not\in (I+J)S_{F \sqcup G_\gamma} \\
& \Longleftrightarrow & & x^\gamma \not\in IS_{F \sqcup G_\gamma} + JS_{F \sqcup G_\gamma} & (\Cref{additionsplitlemma}) \\
& \Longleftrightarrow & & x^\gamma \not\in IS_{F \sqcup G_\gamma} \text{ and } x^\gamma \not\in JS_{F \sqcup G_\gamma} \\
& \Longleftrightarrow & & F \in \Delta_\gamma(I) \text{ and } F \in \Delta_\gamma(J) \\
& \Longleftrightarrow & & F \in \Delta_\gamma(I) \cap \Delta_\gamma(J)
\end{aligned}
\end{equation}
\item By \Cref{reductionlemma2}, under these conditions,  $\Delta_\gamma(I) = \Delta_\alpha(I) \ast \Delta_Y$ and $\Delta_\gamma(J) = \Delta_X \ast \Delta_\beta(J)$. Then
\begin{equation} \notag
\begin{aligned}
\Delta_\gamma(I + J) & = & & \Delta_\gamma(I) \cap \Delta_\gamma(J) \\
& = & & (\Delta_\alpha(I) \ast \Delta_Y) \cap (\Delta_X \ast \Delta_\beta(J)) \\
& = & & (\Delta_\alpha(I) \cap \Delta_X) \ast (\Delta_Y \cap \Delta_\beta(J)) \\
& = & & \Delta_\alpha(I) \ast \Delta_\beta(J).
\end{aligned}
\end{equation}
\end{enumerate}
\end{proof}

\begin{theorem} \label{lem:intersection} 
Let $I$ and $J$ be monomial ideals in $S$ and $\gamma \in \ZZ^n$. Then $\Delta_\gamma(I \cap J) = \Delta_\gamma(I) \cup \Delta_\gamma(J)$.
\end{theorem}

\begin{proof} For any $F \subseteq [n] \setminus G_\gamma$, we have

\begin{equation} \notag
\begin{aligned}
F \in \Delta_\gamma(I \cap J) & \Longleftrightarrow & & x^\gamma \not\in (I \cap J)S_{F \sqcup G_\gamma} \\
& \Longleftrightarrow & & \text{there exists } x^\delta \in k[F \sqcup G_\gamma] \text{ such that } x^\delta x^\gamma \in I \cap J \hspace{0.5cm} (\text{by } \Cref{facelemma}) \\
& \Longleftrightarrow & & \text{there exists } x^\delta \in k[F \sqcup G_\gamma] \text{ such that } x^\delta x^\gamma \not\in  I \text{ or } x^\delta x^\gamma \not\in J \\
& \Longleftrightarrow & & x^\gamma \not\in IS_{F \sqcup G_\gamma} \text{ or } x^\gamma \not\in JS_{F \sqcup G_\gamma} \\
& \Longleftrightarrow & & F \in \Delta_\gamma(I) \text{ or } F \in \Delta_\gamma(J) \\
& \Longleftrightarrow & & F \in \Delta_\gamma(I) \cup \Delta_\gamma(J)
\end{aligned}
\end{equation}
\end{proof}


The case of multiplication of ideals in different polynomial rings reduces to the intersection of those ideals, as follows.

\begin{theorem} \label{multiplying} Let $I \subset A = k[X]$ and $J \subset B = k[Y]$ be monomial ideals and $\gamma = (\alpha,\beta)$. Then $\Delta_\gamma(IJ) = (\Delta_\alpha(I) \ast \Delta_Y) \cup (\Delta_X \ast \Delta_\beta(J))$.
\end{theorem}

\begin{proof} Because $I$ and $J$ share no variables, we have $IJ = I \cap J$, so that 
 
\begin{equation} \notag
\begin{aligned}
\Delta_\gamma(IJ) & = & & \Delta_\gamma(I \cap J) \\
& = & & \Delta_\gamma(I) \cup \Delta_\gamma(J) \hspace{3.5cm} (\text{by } \Cref{lem:intersection}) \\
& = & & (\Delta_\alpha(I) \ast \Delta_Y) \cup (\Delta_X \ast \Delta_\beta(J)) \hspace{1cm} (\text{by } \Cref{reductionlemma2})
\end{aligned}
\end{equation}



\end{proof}

With the description that, for $I \subset A$ and $J \subset B$, $(I + J)^s = \sum_{j=0}^s I^j J^{s-j}$, the degree complex of ordinary powers of a sum of ideals has an explicit decomposition in terms of degree complexes of powers of $I$ and $J$.

\begin{theorem} \label{powerofsum} If $I \subset A$ and $J \subset B$ are monomial ideals and $\gamma = (\alpha,\beta)$, then $$\Delta_\gamma((I+J)^s) = \bigcup_{j=1}^s \Delta_\alpha(I^j) \ast \Delta_\beta(J^{s-j+1}).$$
\end{theorem}

\begin{proof} It is well-known that $(I + J)^s = \sum_{i=0}^s I^i J^{s-i}.$ Thus we have the following first equality from the mentioned fact and

\begin{equation} \notag
\begin{aligned}
\Delta_\gamma((I+J)^s) & = & & \Delta_\gamma\left(\sum_{i=0}^s I^i J^{s-i}\right) \\
& = & & \bigcap_{i=0}^s \Delta_\gamma(I^i J^{s-i}) \hspace{1cm} & (\text{by } \Cref{adding}) \\
& = & & \bigcap_{i=0}^s (\Delta_\alpha(I^i) \ast \Delta_Y) \cup (\Delta_X \ast \Delta_\beta(J^{s-i})) \hspace{1cm} & (\text{by } \Cref{multiplying})
\end{aligned}
\end{equation}

\textit{Claim.} 
$$\bigcap_{i=0}^s (\Delta_\alpha(I^i) \ast \Delta_Y) \cup (\Delta_X \ast \Delta_\beta(J^{s-i})) = \bigcup_{j=1}^s \Delta_\alpha(I^j) \ast \Delta_\beta(J^{s-j+1})$$

\textit{Proof of the Claim.} 
It is clear that 
\begin{equation} \notag
\begin{aligned}
F \in \bigcup_{j=1}^s \Delta_\alpha(I^j) \ast \Delta_\beta(J^{s-j+1}) & \Longleftrightarrow & & F \in \Delta_\alpha(I^j) \ast \Delta_\beta(J^{s-j+1}) \text{ for some } j \in \{1,\ldots,s\}. \\
\end{aligned}
\end{equation}
It follows from \Cref{subsets} and \Cref{convention} that
$$ \Delta_\alpha(I^j) \ast \Delta_\beta(J^{s-j+1}) \subseteq \Delta_\alpha(I^j) \ast \Delta_Y \subseteq \Delta_\alpha(I^i) \ast \Delta_Y \text{ for } j \leq i \leq s,\text{ and } $$ 
$$ \Delta_\alpha(I^j) \ast \Delta_\beta(J^{s-j+1}) \subseteq \Delta_X \ast \Delta_\beta(J^{s-j+1}) \subseteq \Delta_X \ast \Delta_\beta(J^{s-i}) \text{ for } 0 \leq i \leq j - 1.$$ 

Thus, we conclude the following containment 
 $$\Delta_\alpha(I^j) \ast \Delta_\beta(J^{s-j+1}) \subseteq \bigcap_{i=0}^s (\Delta_\alpha(I^i) \ast \Delta_Y) \cup (\Delta_X \ast \Delta_\beta(J^{s-i})).$$

On the other hand, if $F \in \bigcap_{i=0}^s (\Delta_\alpha(I^i) \ast \Delta_Y) \cup (\Delta_X \ast \Delta_\beta(J^{s-i}))$, then $F \in \Delta_\alpha(I^i) \ast \Delta_Y$ for some $i \in \{1,\ldots,s\}$ because $\Delta_X \ast \Delta_\beta(J^{s-s}) = \Delta_X \ast \Delta_\beta(1)$ is the void complex.

Let $j := \min\{i: F \in \Delta_\alpha(I^i) \ast \Delta_Y\} \in \{1,\ldots,s\}$. Then $F  \in \Delta_X \ast \Delta_\beta(J^{s-(j-1)})$ and
\begin{equation} \notag
\begin{aligned}
F & \in & & [\Delta_X \ast \Delta_\beta(J^{s-(j-1)})] \cap [\Delta_\alpha(I^j) \ast \Delta_Y] = \Delta_\alpha(I^j) \ast \Delta_\beta(J^{s-j+1}) \\
& \subseteq & & \bigcup_{j=1}^s \Delta_\alpha(I^j) \ast \Delta_\beta(J^{s-j+1})
\end{aligned}
\end{equation} $\hfill \diamondsuit$

Therefore, we obtain the desired equality from the claim.
\begin{equation} \notag
\begin{aligned}
\Delta_\gamma((I+J)^s) & = & & \bigcap_{i=0}^s (\Delta_\alpha(I^i) \ast \Delta_Y) \cup (\Delta_X \ast \Delta_\beta(J^{s-i})) \\
& = & & \bigcup_{j=1}^s \Delta_\alpha(I^j) \ast \Delta_\beta(J^{s-j+1})
\end{aligned}
\end{equation}

\end{proof}

By \cite[Theorem 3.4]{HNTT19}, the symbolic power of the sum is given by $(I + J)^{(s)} = \sum_{j=0}^s I^{(j)} J^{(s-j)}$. This similar decomposition translates to a similar decomposition of the degree complex.

\begin{theorem} \label{symbolicsum} If $I \subset A$ and $J \subset B$ are monomial ideals and $\gamma = (\alpha, \beta)$, then $$\Delta_\gamma((I+J)^{(s)}) = \bigcup_{j=1}^s \Delta_\alpha(I^{(j)}) \ast \Delta_\beta(J^{(s-j+1)}).$$
\end{theorem}
\begin{proof} By \cite{HNTT19}, for $I \subset A$ and $J \subset B$ monomial ideals, $(I+J)^{(s)} = \sum_{i=0}^s I^{(i)}J^{(s-i)}$. Furthermore, we have that because $I^{(k)} \supseteq I^{(\ell)}$ when $k \leq \ell$, and so $\Delta_\gamma(I^{(k)}) \subseteq \Delta_\gamma(I^{(\ell)})$, in the same way as the ordinary power case. The proof then follows step-for-step identically to the proof of \Cref{powerofsum}.
\end{proof}

\begin{example} Let $I = (x_1x_2, x_2x_3, x_3x_4) \subset A = k[x_1,\ldots,x_4]$ and $J = (x_5x_6x_7,x_7x_8) \subset B = k[x_5,\ldots,x_8]$. Let $\alpha = (0,2,0,0)$ and $\beta = (1,0,0,0)$. For a simplicial complex $\Delta$, say that $\Delta = \langle F_1,F_2,\ldots,F_t \rangle$ if $F_1,\ldots,F_t$ are the \textit{facets} (maximal faces with respect to inclusion) of $\Delta$.

Then
\begin{itemize}
\item $\Delta_\alpha(I) = \{F : x_2^2 \not\in IS_F\} = \langle\{2,4\}\rangle$
\item $\Delta_\alpha(I^2) = \{F : x_2^2 \not\in I^2 S_F\} = \langle \{2,4\} \rangle$
\item $\Delta_\alpha(I^3) = \{F : x_2^2 \not\in I^3 S_F\} = \langle \{2,4\}, \{1,3\}, \{1,4\} \rangle$
\item $\Delta_\beta(J) = \{F: x_5 \not\in J S_F\} = \langle \{5,7\}, \{5,6,8\} \rangle$
\item $\Delta_\beta(J^2) = \{F: x_5 \not\in J^2 S_F\} = \langle \{6,7\},\{5,7\},\{5,6,8\} \rangle$
\item $\Delta_\beta(J^3) = \{F: x_5 \not\in J^3 S_F\} = \langle \{6,7\},\{5,7\},\{5,6,8\} \rangle$
\end{itemize}

Then 
\begin{equation} \notag
\begin{aligned}
\Delta_\gamma((I+J)^3) & = & & \bigcup_{j=1}^3 \Delta_\alpha(I^j) \ast \Delta_\beta(J^{s-j}) \\
& = & & \langle \{2,4,6,7\}, \{2,4,5,7\}, \{2,4,5,6,8\} \rangle \cup \\
& & & \langle \{2,4,6,7\}, \{2,4,5,7\}, \{2,4,5,6,8\} \rangle \cup \\
& & & \langle \{2,4,5,7\}, \{2,4,5,6,8\}, \{1,3,5,7\}, \{1,3,5,6,8\}, \{1,4,5,7\}, \{1,4,5,6,8\} \rangle \\
& = & & \langle \{2,4,5,7\}, \{2,4,6,7\}, \{2,4,5,6,8\}, \{1,3,5,7\}, \\
& & &  \{1,3,5,6,8\}, \{1,4,5,7\}, \{1,4,5,6,8\} \rangle.
\end{aligned}
\end{equation}

\end{example}

With this description of the degree complexes, we can discuss the cohomology. We begin with a lemma to address the common situation of calculating the dimension of the reduced homology of a simplicial join.

\begin{lemma} \label{dimhomjoin} For simplicial complexes $\Gamma_1$ and $\Gamma_2$, $$\dim_k \tilde{H}_p(\Gamma_1 \ast \Gamma_2;k) = \sum_{u + v = p} \dim_k \tilde{H}_u(\Gamma_1;k) \cdot \dim_k \tilde{H}_v(\Gamma_2;k).$$
\end{lemma}

\begin{proof} By K\"unneth formula (cf. \cite{Mun84}), $\tilde{H}_p(\Gamma_1 \ast \Gamma_2;k) = \bigoplus_{u+v=p} \tilde{H}_u(\Gamma_1;k) \otimes \tilde{H}_v(\Gamma_2;k).$ Therefore, because $\tilde{H}_\cdot(\cdot;k) \cong k^\ell$ for some $\ell \in \NN$, and because $k^{\ell_1} \oplus k^{\ell_2} \cong k^{\ell_1 + \ell_2}$ and $k^{\ell_1} \otimes k^{\ell_2} \cong k^{\ell_1 \cdot \ell_2}$, we have

\begin{equation} \notag
\begin{aligned}
\dim_k \tilde{H}_p(\Gamma_1 \ast \Gamma_2;k) & = & & \dim_k \bigoplus_{u+v=p} \tilde{H}_u(\Gamma_1;k) \otimes \tilde{H}_v(\Gamma_2;k) \\
& = & & \sum_{u+v=p} (\dim_k \tilde{H}_u(\Gamma_1;k)) \cdot (\dim_k \tilde{H}_v(\Gamma_2;k))
\end{aligned}
\end{equation}

\end{proof} 

Using \Cref{takayama} and K\"unneth formula, we may compute the local $\ZZ^s$-graded cohomology of $S/(I + J)$ explicitly in terms of the local graded $\ZZ^s$-graded cohomology of $A/I$ and of $B/J$.

\begin{theorem} \label{cohomologyofsum} If $I \subset A = k[x_1,\ldots,x_m]$ and $J \subset B = k[x_{m+1},\ldots,x_n]$ are monomial ideals, $\mm = (x_1,\ldots,x_m)$ and $\nn = (x_{m+1},\ldots,x_n)$, $S = A \otimes_k B$, and $\gamma = (\alpha, \beta)$, then 
$$\dim_k H_{\mm+\nn}^p(S/(I+J))_\gamma = \sum_{u + v = p + 1} \dim_k H_\mm^u(A/I)_\alpha \cdot \dim_k H_\nn^v(B/J)_\beta$$
\end{theorem}

\begin{proof} By \Cref{takayama}, if $G_\gamma \in \Delta_\ob(I+J)$, we have

\begin{equation} \notag
\begin{aligned}
\dim_k H_{\mm+\nn}^p(S/(I+J))_\gamma & = & & \dim_k \tilde{H}_{p - |G_\gamma| - 1}(\Delta_\gamma(I+J);k) \\ 
& = & & \dim_k \tilde{H}_{p - |G_\gamma| - 1}(\Delta_\alpha(I) \ast \Delta_\beta(J);k) \\
& = & & \sum_{u' + v' = p - |G_\gamma| - 1} \dim_k \tilde{H}_{u'}(\Delta_\alpha(I);k) \cdot \dim_k \tilde{H}_{v'}(\Delta_\beta(J);k),
\end{aligned}
\end{equation}

where the last equality is due to \Cref{dimhomjoin}.  

Letting $u' = u - |G_\alpha| - 1$ and $v' = v - |G_\beta| - 1$, we have
\begin{equation} \notag
\begin{aligned}
u' + v' = p - |G_\gamma| - 1 & \Longleftrightarrow & & u - |G_\alpha| - 1 + v - |G_\beta| - 1 = p - |G_\gamma| - 1 \\
& \Longleftrightarrow & & u + v = p + 1
\end{aligned}
\end{equation}

Then we have

\begin{equation} \notag
\begin{aligned}
\dim_k H_{\mm+\nn}^p(S/(I+J))_\gamma & = & & \sum_{u + v = p + 1} \dim_k \tilde{H}_{u - |G_\alpha| - 1}(\Delta_\alpha(I);k) \cdot \dim_k \tilde{H}_{v - |G_\beta| - 1}(\Delta_\beta(J);k) \\
& = & & \sum_{u + v = p + 1} \dim_k H_\mm^u(A/I)_\alpha \cdot \dim_k H_\nn^v(B/J)_\beta,
\end{aligned}
\end{equation}

where the last equality is again due to \Cref{takayama}.

\end{proof}

Likewise, we may compute the cohomology of $S/(IJ)$ in terms of the cohomologies of $A/I$ and $B/J$, due to an isomorphism derived in the Mayer-Vietoris sequence between $\tilde{H}_p(S/(IJ))$ and $\tilde{H}_{p-1}(S/(I+J))$. The author wishes to thank Selvi Kara Beyarslan for presenting this isomorphism.

\begin{theorem} \label{cohomologyofproduct} If $I \subset A$ and $J \subset B$ are monomial ideals, $S = A \otimes_k B$, and $\gamma = (\alpha,\beta)$, then

$$\dim_k H_{\mm+\nn}^p(S/(IJ))_\gamma = \sum_{u + v = p} \dim_k H_\mm^u(A/I)_\alpha \cdot \dim_k H_\nn^v(B/J)_\beta$$
\end{theorem}

\begin{proof} By \Cref{multiplying}, 
$$\Delta_\gamma(IJ) = (\Delta_\alpha(I) \ast \Delta_Y) \cup (\Delta_X \ast \Delta_\beta(J)).$$

Furthermore,
\begin{equation} \notag
\begin{aligned} 
(\Delta_\alpha(I) \ast \Delta_Y) \cap (\Delta_X \ast \Delta_\beta(J)) & = & & (\Delta_\alpha(I) \cap \Delta_X) \ast (\Delta_Y \cap \Delta_\beta(J)) \\
& = & & \Delta_\alpha(I) \ast \Delta_\beta(J) \\
& = & & \Delta_\gamma(I + J)
\end{aligned}
\end{equation}

Therefore, via a Mayer-Vietoris sequence (cf. \cite[pg.22]{Sta83}), we have the following long exact sequence of reduced homologies.

$$\cdots \to \tilde{H}_p(\Delta_\alpha(I) \ast \Delta_Y) \oplus \tilde{H}_p(\Delta_X \ast \Delta_\beta(J)) \to \tilde{H}_p(\Delta_\gamma(IJ)) \to \tilde{H}_{p-1}(\Delta_\gamma(I + J))$$
$$ \to \tilde{H}_{p-1}(\Delta_\alpha(I) \ast \Delta_Y) \oplus \tilde{H}_{p-1}(\Delta_X \ast \Delta_\beta(J)) \to \cdots$$

By K\"unneth formula, we have $\tilde{H}_p(\Delta_\alpha(I) \ast \Delta_Y) = \bigoplus_{u + v = p} \tilde{H}_u(\Delta_\alpha(I)) \otimes \tilde{H}_v(\Delta_Y)$. Furthermore, because a simplex is contractible, $\tilde{H}_v(\Delta_Y) = 0$ for all $v$, so that $\bigoplus_{u + v = p} \tilde{H}_u(\Delta_\alpha(I)) \otimes \tilde{H}_v(\Delta_Y) = 0$ for any $p$. Similarly, $\tilde{H}_p(\Delta_X \ast \Delta_\beta(J)) = 0$ for all $p$. Therefore, for each $p$, the Mayer-Vietoris sequence reduces to

$$\cdots \to 0 \to \tilde{H}_p(\Delta_\gamma(IJ)) \tilde{\to} \tilde{H}_{p-1}(\Delta_\gamma(I+J)) \to 0 \to \cdots$$

Therefore 

\begin{equation} \notag
\begin{aligned}
\dim_k H_{\mm + \nn}^p(S/(IJ))_\gamma & = & & \dim_k \tilde{H}_{p - |G_\gamma| - 1}(\Delta_\gamma(IJ);k) \\
& = & & \dim_k \tilde{H}_{p - |G_\gamma| - 2}(\Delta_\gamma(I + J);k) \\
& = & & \dim_k H_{\mm + \nn}^{p - 1}(S/(I+J))_\gamma \\
& = & & \sum_{u + v = p} (\dim_k H_\mm^u(A/I)_\alpha) \cdot (\dim_k H_\nn^v(B/J)_\beta) \hspace{1cm} (\text{by } \Cref{cohomologyofsum})
\end{aligned}
\end{equation}

\end{proof}

Using a collection of Mayer-Vietoris sequences, the dimension of the local cohomology modules of $\Delta_i := \bigcup_{j=i}^s \Delta_\alpha(I^j) \ast \Delta_\beta(J^{s-j+1})$ can be computed in terms of the dimensions of $H_\mm^u(A/I^{i})$, $H_\nn^v(B/J^{j})$, and $\tilde{H}_p(\Delta_{i+1})$. Because $\tilde{H}_p(\Delta_s)$ can also be computed in terms of $H_\mm^u(A/I^i)$ and $H_\nn^v(B/J^j)$, the dimension of the graded cohomology of $\Delta_1 = \Delta_\gamma((I+J)^s)$ could be computed from these long exact sequences. 

\begin{theorem} \label{lescohom} There is a collection of long exact sequences which relate $H_{\mm+\nn}^p(S/(I+J)^s)_\gamma$ to the modules $H_\mm^u(A/I^i)_\alpha$ and $H_\nn^v(B/J^j)_\beta$ and a collection which relate $H_{\mm+\nn}^p(S/(I+J)^{(s)})_\gamma$ to the modules $H_\mm^u(A/I^{(i)})_\alpha$ and $H_\nn^v(B/J^{(j)})_\beta$, with $1 \leq i,j \leq s$.
\end{theorem}

\begin{proof} Let $\Delta_i := \bigcup_{j = i}^s \Delta_\alpha(I^j) \ast \Delta_\beta(J^{s-j+1})$. Then by \Cref{powerofsum}, $\Delta_\gamma((I+J)^s) = \Delta_1$. By definition, $\Delta_i = (\Delta_\alpha(I^i) \ast \Delta_\beta(J^{s-i+1})) \cup \Delta_{i+1}$. Moreover, 

\begin{equation} \notag
\begin{aligned}
(\Delta_\alpha(I^i) \ast \Delta_\beta(J^{s-i+1})) \cap \Delta_{i+1} & = & & (\Delta_\alpha(I^i) \ast \Delta_\beta(J^{s-i+1})) \cap \bigg( \bigcup_{j=i+1}^s \Delta_\alpha(I^j) \ast \Delta_\beta(J^{s-j+1}) \bigg) \\
& = & & \bigcup_{j=i+1}^s (\Delta_\alpha(I^i) \cap \Delta_\alpha(I^j)) \ast (\Delta_\beta(J^{s-i+1}) \cap \Delta_\beta(J^{s-j+1})) \\
& = & & \bigcup_{j=i+1}^s (\Delta_\alpha(I^i) \ast \Delta_\beta(J^{s-j+1}) \\
& = & & \Delta_\alpha(I^i) \ast \Delta_\beta(J^{s-i})
\end{aligned}
\end{equation}

Therefore, we may construct the following Mayer-Vietoris sequences for $i = 1,\ldots,s-1$.

$$\cdots \to \tilde{H}_p(\Delta_\alpha(I^i) \ast \Delta_\beta(J^{s-i})) \overset{f_\ast}{\to} \tilde{H}_p(\Delta_\alpha(I^i) \ast \Delta_\beta(J^{s-i+1})) \oplus \tilde{H}_p(\Delta_{i+1}) \to \tilde{H}_p(\Delta_i)$$
$$\to \tilde{H}_{p-1}(\Delta_\alpha(I^i) \ast \Delta_\beta(J^{s-i})) \overset{f_\ast}{\to} \tilde{H}_{p-1}(\Delta_\alpha(I^i) \ast \Delta_\beta(J^{s-i+1})) \oplus \tilde{H}_{p-1}(\Delta_{i+1}) \to \cdots$$

where the map $f_\ast = (\iota_\ast^1,\iota_\ast^2)$ is the map induced by the inclusions $\iota^1: \Delta_\alpha(I^i) \ast \Delta_\beta(J^{s-i}) \to \Delta_\alpha(I^i) \ast \Delta_\beta(J^{s-i+1})$ and $\iota^2: \Delta_\alpha(I^i) \ast \Delta_\beta(J^{s-i}) \to \Delta_{i+1}$.

Furthermore, via \Cref{adding} and \Cref{cohomologyofsum},
\begin{equation} \notag
\begin{aligned}
\dim_k\tilde{H}_p(\Delta_s) & = & & \dim_k \tilde{H}_p(\Delta_\alpha(I^s) \ast \Delta_\beta(J)) \\
& = & & \dim_k \tilde{H}_p(\Delta_\gamma(I^s + J)) \\
& = & & \sum_{u + v = p + 1} (\dim_k H_\mm^u(A/I^s)_\alpha) \cdot (\dim_k H_\nn^v(B/J)_\beta)
\end{aligned}
\end{equation}

Via \Cref{adding}, \Cref{takayama}, and \Cref{cohomologyofsum},


\begin{equation} \notag
\begin{aligned}
\dim_k \tilde{H}_p(\Delta_\alpha(I^i) \ast \Delta_\beta(J^j)) & = & & \dim_k \tilde{H}_p(\Delta_\gamma(I^i + J^j)) \\
& = & & \dim_k H_{\mm + \nn}^{p + |G_\gamma| + 1}(S/(I^i + J^j))_\gamma \\
& = & & \sum_{u + v = p + |G_\gamma| + 2} \dim_k H_\mm^u(A/I^i)_\alpha \cdot \dim_k H_\nn^v(B/J^j)_\beta
\end{aligned}
\end{equation}

so that the entries of the Mayer-Vietoris sequence, besides $\tilde{H}_p(\Delta_{i+1})$ and $\tilde{H}_p (\Delta_i)$, can be computed explicitly in terms of the modules $H_\mm^u(A/I^i)_\alpha$, $H_\mm^v(B/J^{s-i})_\beta$, and $H_\mm^v(B/J^{s-i+1})_\beta$. The Mayer-Vietoris sequence also yields the short exact sequence
$$0 \to \text{coker} f_\ast \to \tilde{H}_p(\Delta_i) \to \ker f_\ast \to 0,$$
where the first $f_\ast: \tilde{H}_p(\Delta_\alpha(I^i) \ast \Delta_\beta(J^{s-i})) \overset{f_\ast}{\to} \tilde{H}_p(\Delta_\alpha(I^i) \ast \Delta_\beta(J^{s-i+1})) \oplus \tilde{H}_p(\Delta_{i+1})$ and the second $f_\ast: \tilde{H}_{p-1}(\Delta_\alpha(I^i) \ast \Delta_\beta(J^{s-i})) \overset{f_\ast}{\to} \tilde{H}_{p-1}(\Delta_\alpha(I^i) \ast \Delta_\beta(J^{s-i+1})) \oplus \tilde{H}_{p-1}(\Delta_{i+1})$. With enough information about this induced homomorphism, $\tilde{H}_p(\Delta_i)$ could be computed in terms of the above local cohomology modules and $\tilde{H}_p(\Delta_{i+1})$. Because $\tilde{H}_p(\Delta_s)$ is computed as well, we have

\begin{equation} \notag
\begin{aligned}
\dim_k H_{\mm + \nn}^p(S/(I+J)^s)_\gamma & = & & \dim_k \tilde{H}_{p - |G_\gamma| - 1}(\Delta_\gamma((I+J)^s)) \\ & = & & \dim_k \tilde{H}_{p - |G_\gamma| - 1}(\Delta_1)
\end{aligned}
\end{equation}
can be computed in terms of these modules as well, given the necessary information about $f_\ast$.

The proof of the symbolic power case is identical to the ordinary power case, via the decomposition given by \Cref{symbolicsum}.

\end{proof}


\section{Degree Complexes of Powers of Fiber Products of Ideals} \label{sec.powerfiber}

In this section, we consider ordinary and symbolic powers of fiber products of squarefree monomial ideals. These ideals have the nice property that their degree complexes will split into a disjoint union of two complexes (when both are nonempty). This makes it possible to compute the reduced homology of the degree complex, and thus the cohomology of the $\ZZ^s$-graded quotient ring, straightforwardly.

\begin{definition} Given monomial ideals $I \subseteq A = k[x_1,\ldots,x_m]$ and $J \subseteq B = k[x_{m+1},\ldots,x_n]$ in different polynomial rings, the \textit{fiber product} of $I$ and $J$ is defined to be $I + J + \mm \nn$, where $\mm = (x_1,\ldots,x_m)$ and $\nn = (x_{m+1},\ldots,x_n)$.
\end{definition}



In order to prove the decomposition of $\Delta_\gamma((I + J + \mm\nn)^s)$, I will first prove a few helpful lemmas.

The following lemma will help in computing $(I+J+\mm\nn)^s S_F$ when $F$ is nonempty.

\begin{lemma} \label{maximalideal}
Let $\mm = (x_1,\ldots,x_m)$. If $F \cap \{1,\ldots,m\} \neq \emptyset$, then $\mm^i S_F = S_F$ for any $i$.
\end{lemma}
\begin{proof}
Suppose $j \in F \cap \{1,\ldots,m\}$. Then because $x_j^i \in k[F]$, we have
\begin{equation} \notag
\begin{aligned}
x^\gamma \in \mm^i S_F & \Longleftrightarrow & & x_j^i x^\gamma \in \mm^i S_F & & (\text{Lemma } \ref{freeexponents}) \\
& \Longleftrightarrow & & x^\gamma \in S_F, \text{ because } x_j^i \in \mm^i
\end{aligned}
\end{equation}
\end{proof}

The following lemma shows that any face must be contained in either $\{1,\ldots,m\}$ or $\{m+1,\ldots,n\}$, which will create a disjoint union of degree complexes, each on one of these sets.

\begin{lemma} \label{fibprodlem1} 
Let $F \subseteq [n] $ and  $\gamma = (\alpha,\beta) \in \ZZ^m \times \ZZ^{n-m}.$ For $F \in \Delta_\gamma((I + J + \mm\nn)^s)$ or in $\Delta_\gamma((I+J+\mm\nn)^{(s)})$, we have $F \cup G_\gamma \subseteq \{1,\ldots,m\}$ or $F \cup G_\gamma \subseteq \{m+1,\ldots,n\}$.
\end{lemma}

\begin{proof} Suppose  $\{i,j\} \subseteq F \cup G_\gamma$ for some $i \in \{1,\ldots,m\}$ and $j \in \{m+1,\ldots,n\}.$ Since $x_i x_j \in \mm\nn \subseteq (I + J + \mm\nn)$ and $(x_i x_j)^{-1} \in S_{F \cup G_\gamma}$, one can see that, for any $\gamma \in \ZZ^n$, $$x^\gamma = \prod_{k \not\in G_\gamma} x_k^{\gamma_k} (x_i x_j)^s \cdot (x_i x_j)^{-s} \prod_{k \in G_\gamma} x_k^{\gamma_k} \in (I + J + \mm\nn)^s S_{F \cup G_\gamma} \subseteq (I + J + \mm\nn)^{(s)} S_{F \cup G_\gamma}.$$ Thus $F \notin \Delta_\gamma((I + J + \mm\nn)^s)$ and $F \notin \Delta_\gamma((I+J+\mm\nn)^{(s)})$, for any $F$ for which $(F \cup G_\gamma)$ has elements in both $\{1,\ldots,m\}$ and $\{m+1,\ldots,n\}$. 
\end{proof}

The following lemma describes the degree complex obtained on each of the sets $\{1,\ldots,m\}$ and $\{m+1,\ldots,n\}$.

\begin{lemma} \label{fibprodlem2} 
Let $F \subseteq [n] $ and  $\gamma = (\alpha,\beta) \in \ZZ^m \times \ZZ^{n-m}.$ Suppose $F \cup G_\gamma \neq \emptyset$. 
\begin{enumerate}
\item If $F \cup G_\gamma \subseteq \{1,\ldots,m\},$ then 
$$F \in \Delta_\gamma((I+J+\mm\nn)^s) \iff F \in \Delta_\alpha(I^{s-|\beta|}).$$ 
\item If $F \cup G_\gamma \subseteq \{m+1,\ldots,n\}$, then 
$$F \in \Delta_\gamma((I+J+\mm\nn)^s) \iff F \in \Delta_\beta(J^{s-|\alpha|}).$$
\end{enumerate}

\end{lemma}

\begin{proof} It suffices to prove (1), as (2) follows similarly from symmetry. Suppose $F \cup G_\gamma \subseteq \{1,\ldots,m\}.$ It follows immediately that $G_\gamma = G_\alpha$ and $\beta_j \geq 0$ for each $j$.

Because $F \cup G_\alpha \neq \emptyset$, we have $\mm^i S_{F \cup G_\alpha} = S_{F \cup G_\alpha}$ by \Cref{maximalideal}. Furthermore, because $J \subseteq \nn$, we have $J + \nn = \nn$. Therefore

\begin{equation} \notag
\begin{aligned}
(I + J + \mm\nn)^s S_{F \cup G_\gamma} & = & & \sum_{i=0}^s (I + J)^{s-i} (\mm\nn)^i S_{F \cup G_\alpha} \\
& = & & \sum_{i=0}^s (I + J)^{s-i} \nn^i S_{F \cup G_\alpha} \\
& = & & (I + J + \nn)^s S_{F \cup G_\alpha} \\
& = & & (I + \nn)^s S_{F \cup G_\alpha} \\
& = & & \sum_{i=0}^s I^{s-i} S_{F \cup G_\alpha} \cdot \nn^i S_{F \cup G_\alpha} \\
& = & & \sum_{i=0}^s (I^{s-i} S_{F \cup G_\alpha} \cap \nn^i)
\end{aligned}
\end{equation}

In this case,
\begin{equation} \notag
\begin{aligned}
F \in \Delta_\gamma((I + J + \mm\nn)^s) & \Longleftrightarrow & & x^\gamma \not\in (I + J + \mm\nn)^s S_{F \cup G_\gamma} \\
& \Longleftrightarrow & & x^\gamma \not\in \sum_{i=0}^s (I^{s-i} S_{F \cup G_\alpha} \cap \nn^i) \\
& \Longleftrightarrow & & \text{for each } i = 0,\ldots,s, \text{ either } x^\alpha \not\in I^{s-i} S_{F \cup G_\alpha} \text{ or }  x^\beta \not\in \nn^i
\end{aligned}
\end{equation}

Observe that, when $G_\beta \neq 0$, $x^\beta \not\in \nn^i $ is equivalent to $ i > |\beta|$. If $|\beta| \geq s$, then $x^\beta \in \nn^i$ for all $i = 0,\ldots,s$, and so we have $F$ cannot be in $\Delta_\gamma((I+J+\mm\nn)^s)$. When $x^\beta \not\in \nn^i,$ we have $I^s \subseteq \cdots \subseteq I^{s-|\beta|}$, which yields $\sum_{i=0}^{|\beta|} I^{s-i} = I^{s - |\beta|}.$ Thus we can continue to simplify the last statement in the above sequence of equivalences as follows.

\begin{equation} \notag
\begin{aligned}
& \Longleftrightarrow & & x^\alpha \not\in \sum_{i=0}^{|\beta|} I^{s-i} S_{F \cup G_\alpha} \text{ and } |\beta| < s\\
& \Longleftrightarrow & & x^\alpha \not\in I^{s-|\beta|} S_{F \cup G_\alpha} \text{ and } |\beta| < s \\
& \Longleftrightarrow & & F \in \Delta_\alpha(I^{s-|\beta|}) \text{ and } |\beta| < s
\end{aligned}
\end{equation}

Likewise, if $F \cup G_\gamma \subseteq \{m+1,\ldots,n\}$, then $$F \in \Delta_\gamma((I + J + \mm\nn)^s) \Longleftrightarrow F \in \Delta_\beta(J^{s-|\alpha|}) \text{ and } |\alpha| < s.$$
\end{proof}

With this proved, we may now state the degree complex decomposition of the power of the fiber product.

\begin{theorem} \label{ordinaryfiberproduct} Let $I \subset A$ and $J \subset B$ be squarefree monomial ideals, $\gamma = (\alpha,\beta) \in \ZZ^m \times \ZZ^{n-m}$, and $s \geq 1$. Let
$$\mathcal{A} = \begin{cases} \{\emptyset \subsetneq F \in \Delta_\alpha(I^{s-|\beta|})\}, & G_\beta = \emptyset, |\beta| < s \\ \emptyset, & G_\beta \neq \emptyset \text{ or } |\beta| \geq s \end{cases} \text{ be nonempty faces of } \Delta_\alpha(I^{s-|\beta|});$$ $$\mathcal{B} = \begin{cases} \{\emptyset \subsetneq F \in \Delta_\beta(J^{s-|\alpha|})\}, & G_\alpha = \emptyset, |\alpha| < s \\ \emptyset, & G_\alpha \neq \emptyset \text{ or } |\alpha| \geq s \end{cases} \text{ be nonempty faces of } \Delta_\beta(J^{s-|\alpha|}).$$ Then the nonempty faces of $\Delta_\gamma((I+J+\mm\nn)^s)$ are given by $$\{\emptyset \subsetneq F \in \Delta_\gamma((I+J+\mm\nn)^s)\} = \mathcal{A} \sqcup \mathcal{B}.$$


	


\end{theorem}

\begin{proof} 
Suppose $F \in \Delta_\gamma((I+J+\mm\nn)^s)$ and that $F \sqcup G_\gamma \neq \emptyset$. It follows from \Cref{fibprodlem1} that either $F \sqcup G_\gamma \subseteq \{1,\ldots,m\}$ or $F \sqcup G_\gamma \subseteq \{m+1,\ldots,n\}$. If $F \sqcup G_\gamma \subseteq \{1,\ldots,m\}$, then $F \sqcup G_\gamma = F \sqcup G_\alpha$, and by \Cref{fibprodlem2}, $F \in \Delta_\alpha(I^{s-|\beta|})$, $|\beta| < s$ and $G_\beta = \emptyset$. Similarly, if $F \sqcup G_\gamma \subseteq \{m+1,\ldots,n\}$, then $F \sqcup G_\gamma = F \sqcup G_\beta$, and by \Cref{fibprodlem2}, $F \in \Delta_\beta(J^{s-|\alpha|})$, $|\alpha| < s$, and $G_\alpha = \emptyset$. 
\\
On the other hand, if $\emptyset \neq F \in \Delta_\alpha(I^{s-|\beta|})$, then $G_\beta = \emptyset$ and $|\beta| < s$. Then $F \sqcup G_\gamma \subseteq \{1,\ldots,m\}$. so that $F \sqcup G_\alpha = F \sqcup G_\gamma$. By \Cref{fibprodlem2}, $F \in \Delta_\gamma((I+J+\mm\nn)^s)$. Likewise, if $F \in \Delta_\beta(J^{s-|\alpha|})$, then $G_\alpha = \emptyset$, $|\alpha| < s$, and $F \in \Delta_\gamma((I+J+\mm\nn)^s)$.
\end{proof}


\begin{observation} It is possible that $\Delta_\gamma((I+J+\mm\nn)^s)$ is the irrelevant complex and $\Delta_\alpha(I^{s-|\beta|}) \cup \Delta_\beta(J^{s-|\alpha|})$ is the void complex. For instance, let $I = (x_1 x_2) \subset k[x_1, x_2]$, $J = (x_3 x_4) \subset k[x_3, x_4]$, $\mm\nn = (x_1 x_3, x_1 x_4, x_2 x_3, x_3 x_4)$, and $\gamma = (\alpha,\beta) = (1,1,1,1)$. Because $x^\gamma = x_1 x_2 x_3 x_4 \not\in (I + J + \mm\nn)^3$, it follows that $\emptyset \in \Delta_\gamma((I+J+\mm\nn)^3)$. However, because $x^\alpha = x_1 x_2 \in I = I^{3 - |\beta|}$ and $x^\beta = x_3 x_4 \in J = J^{3 - |\alpha|}$, it follows that $\emptyset \not\in \Delta_\alpha(I^{s-|\beta|}) \cup \Delta_\beta(J^{s-|\alpha|})$. Therefore this theorem cannot be simplified to saying that the two complexes are equal.

\end{observation}

Now, let us examine $(I + J + \mm\nn)^{(s)}$. Recall that $$(I + J + \mm\nn)^{(s)} = \bigcap_{\pp \in \Ass(I + J + \mm\nn)} \pp^s.$$ In order to examine $(I + J + \mm\nn)^{(s)}$, we should then determine these associated primes.

\begin{lemma} \label{assprimesSFP} Let $I \subset A$ and $J \subset B$ be squarefree monomial ideals, $\mm = (x_1,\ldots,x_m)$, and $\nn = (x_{m+1},\ldots,x_n)$. Then $$\Ass(I + J + \mm\nn) = \{\qq + \nn: \qq \in \Ass(I)\} \cup \{\mm + \rr: \rr \in \Ass(J)\}.$$
\end{lemma}

\begin{proof} For $\qq \in \Ass(I)$, let $M \in A$ be the monomial for which $\qq$ is the annihilator. Then because $x_i x_j \in I$ for all $x_i \in A$ and $x_j \in B$, it follows that $x_j M \in \mm\nn \subset I + J + \mm\nn$ for all $x_j \in B$. Therefore $(\qq + \nn) = \text{Ann}_{I + J + \mm\nn}(M)$, so that $\qq + \nn \in \Ass(I + J + \mm\nn)$. Likewise, for any $\rr \in \Ass(J)$, it follows that $\mm + \rr = \text{Ann}_{I + J + \mm\nn}(M)$, and so $\mm + \rr \in \Ass(I + J + \mm\nn)$.

On the other hand, for $\pp \in \Ass(I + J + \mm\nn)$, if $0 \neq M \in S/(I + J + \mm\nn)$ is the monomial that $\pp$ annihilates, $\supp(M) \subseteq \{1,\ldots,m\}$ or $\supp(M) \subseteq \{m+1,\ldots,n\}$ (as otherwise something in $\mm\nn$ divides $M$). If $\supp(M) \subseteq \{1,\ldots,m\}$, then

\begin{equation} \notag
\begin{aligned}
\pp \cap \{x_1,\ldots,x_m\} & = & & \{x_j : x_j M \in I + J + \mm\nn\} \cap \{x_1,\ldots,x_m\} \\
& = & & \{x_j : x_j M \in I\} \in \Ass(I)
\end{aligned}
\end{equation} 

and

\begin{equation} \notag
\begin{aligned}
\pp \cap \{x_{m+1},\ldots,x_n\} & = & & \{x_j : x_j M \in I + J + \mm\nn\} \cap \{x_{m+1},\ldots,x_n\} \\
& = & & \{x_{m+1},\ldots,x_n\} = \nn,
\end{aligned}
\end{equation}

so that $\pp = \qq + \nn$ for some $\qq \in \Ass(I)$. Likewise, if $\supp(M) \subseteq \{m+1,\ldots,n\}$, then $\pp = \mm + \rr$ for some $\rr \in \Ass(J)$. This completes the proof.
\end{proof}

With this description of the associated primes, we may work with the symbolic power. As before, the degree complex is a disjoint union of degree complexes on $\{1,\ldots,m\}$ and $\{m+1,\ldots,n\}$. The following lemma describes each of these degree complexes.

\begin{lemma} \label{reducingsymbolic} Let $I \subset A$ and $J \subset B$ be squarefree monomial ideals, $S = A \otimes_k B$, and $\gamma = (\alpha,\beta) \in \ZZ^m \times \ZZ^{n-m}$.
\begin{enumerate}
\item If $\emptyset \neq F \sqcup G_\gamma \subseteq \{1,\ldots,m\}$, then $x^\gamma \in (I + J + \mm\nn)^{(s)} S_{F \sqcup G_\gamma}$ if and only if $|\beta| \geq s$ or $x^\alpha \in I^{(s - |\beta|)}S_{F \sqcup G_\gamma}$.
\item If $\emptyset \neq F \sqcup G_\gamma \subseteq \{m+1,\ldots,n\}$, then $x^\gamma \in (I + J + \mm\nn)^{(s)} S_{F \sqcup G_\gamma}$ if and only if $|\alpha| \geq s$ or $x^\beta \in J^{(s - |\alpha|)}S_{F \sqcup G_\gamma}$.
\end{enumerate}
\end{lemma} 

\begin{proof} It suffices to prove one, as the other follows from symmetry. By \Cref{assprimesSFP}, we can write

\begin{equation} \notag
\begin{aligned}
(I + J + \mm\nn)^{(s)} & =  & & \bigcap_{\pp \in \Ass((I + J + \mm\nn)} \pp^s \\
& = & & \bigcap_{\qq \in \Ass(I)} (\qq + \nn)^s \cap \bigcap_{\rr \in \Ass(J)} (\mm + \rr)^s
\end{aligned}
\end{equation}

Let $\emptyset \neq  F \sqcup G_\gamma \subseteq \{1,\ldots,m\}$. Then $x^\gamma \in \mm^s S_{F \sqcup G_\gamma} \subseteq (\mm + \rr)^s S_{F \sqcup G_\gamma}$ for any $s$, so $x^\gamma$ is in this intersection iff $x^\gamma \in \bigcap_{\qq \in \Ass(I)} (\qq + \nn)^s$. Furthermore, because $(\qq + \nn)^s = \sum_{i=0}^s \qq^{s-i} \nn^i= \sum_{i=0}^s \qq^{s-i} \cap \nn^i$, and because $x^\gamma \in \nn^i$ iff $x^\beta \in \nn^i$, which is true if and only if $|\beta| \leq i$, we have $x^\gamma \in (\qq + \nn)^s$ iff $x^\gamma \in \sum_{i=0}^{|\beta|} \qq^{s-i}$ or $|\beta| \geq s$, which is true iff $x^\gamma \in \qq^{s - |\beta|}$ or $|\beta| \geq s$. Therefore, for $\emptyset \neq F \sqcup G_\gamma$,

\begin{equation} \notag
\begin{aligned}
x^\gamma \in (I + J + \mm\nn)^{(s)} & \Longleftrightarrow & & x^\gamma \in \bigcap_{\qq \in \Ass(I)} (\qq + \nn)^s S_{F \sqcup G_\gamma} \\
& \Longleftrightarrow & & x^\gamma \in \bigcap_{\qq \in \Ass(I)} S_{F \sqcup G_\gamma} = I^{(s - |\beta|)} S_{F \sqcup G_\gamma} \text{ or } |\beta| \geq s \\
& \Longleftrightarrow & & x^\alpha \in I^{(s - |\beta|)} S_{F \sqcup G_\gamma} \text{ or } |\beta| \geq s
\end{aligned}
\end{equation}

This concludes the proof of the first part. The same proof gives the second part.

\end{proof}

With this in place, we may now give the degree complex decomposition of the symbolic powers of the fiber product.

\begin{theorem} \label{symbolicfiberproduct} Let $I \subset A$ and $J \subset B$ be squarefree monomial ideals, $\gamma = (\alpha,\beta) \in \ZZ^m \times \ZZ^{n-m}$, and $s \geq 1$. Let $$\mathcal{A'} = \begin{cases} \{\emptyset \subsetneq F \in \Delta_\alpha(I^{(s-|\beta|)})\}, & G_\beta = \emptyset, |\beta| < s \\ \emptyset, & G_\beta \neq \emptyset \text{ or } |\beta| \geq s \end{cases} \text{ be nonempty faces of } \Delta_\alpha(I^{(s-|\beta)|});$$ $$\mathcal{B'} = \begin{cases} \{\emptyset \subsetneq F \in \Delta_\beta(J^{(s-|\alpha|)})\}, & G_\alpha = \emptyset, |\alpha| < s \\ \emptyset, & G_\alpha \neq \emptyset \text{ or } |\alpha| \geq s \end{cases} \text{ be nonempty faces of } \Delta_\beta(J^{(s-|\alpha|)}).$$ Then the nonempty faces of $\Delta_\gamma((I+J+\mm\nn)^{(s)})$ are given by
	
	$$\{\emptyset \subsetneq F \in \Delta_\gamma((I+J+\mm\nn)^{(s)})\} = \mathcal{A'} \sqcup \mathcal{B'}$$





\end{theorem}

\begin{proof} By \Cref{fibprodlem1}, $F \in \Delta_\gamma((I+J+\mm\nn)^{(s)})$ implies that $F \sqcup G_\gamma \subseteq \{1,\ldots,m\}$ or $F \sqcup G_\gamma \subseteq \{m+1,\ldots,n\}$ for any $\gamma$. Assuming $F \sqcup G_\gamma \neq \emptyset$, if $F \subseteq \{1,\ldots,m\}$, then $G_\beta = \emptyset $, and by \Cref{reducingsymbolic},
\begin{equation} \notag
\begin{aligned}
F \in \Delta_\gamma((I + J + \mm\nn)^{(s)}) & \Longleftrightarrow & & x^\gamma \not\in (I + J + \mm\nn)^{(s)} S_{F \sqcup G_\gamma} \\
& \Longleftrightarrow & & x^\alpha \not\in I^{(s - |\beta|)} S_{F \sqcup G_\gamma} = I^{(s - |\beta|)} S_{F \sqcup G_\alpha} \text{ and } |\beta| < s \\
& \Longleftrightarrow & & F \in \Delta_\alpha(I^{(s - |\beta|)}), G_\beta = \emptyset, \text{ and } |\beta| < s
\end{aligned}
\end{equation}

Likewise, if $F \subseteq \{m+1,\ldots,n\}$, then $G_\alpha = \emptyset$, $|\alpha| < s$, and $F \in \Delta_\beta(J^{(s-|\alpha|)})$

On the other hand, if $\emptyset \neq F \in \Delta_\alpha(I^{(s - |\beta|)})$, then $|\beta|$ must be less than $s$, $G_\beta$ must be empty, and $F \sqcup G_\alpha = F \sqcup G_\gamma \subseteq \{1,\ldots,m\}$. Then by \Cref{reducingsymbolic}, we have again that $F \in \Delta_\gamma((I + J + \mm\nn)^{(s)})$.

Therefore, the nonempty faces of $\Delta_\gamma((I + J + \mm\nn)^{(s)})$ are precisely the nonempty faces of either $\Delta_\alpha(I^{(s-|\beta|)})$ or $\Delta_\beta(J^{(s-|\alpha|)})$.

\end{proof}

With the decompositions of \Cref{ordinaryfiberproduct} and \Cref{symbolicfiberproduct}, we can compute the dimensions of the local cohomology modules explicitly in terms of those of $A/I^t$ and $B/J^t$ (respectively, $A/I^{(t)}$ and $B/J^{(t)}$).

\begin{theorem} \label{cohomologyfiberproduct} Let $I \subset A = k[x_1,\ldots,x_m]$, $J \subset B = k[x_{m+1},\ldots,x_n]$, $S = A \otimes_k B$, $\mm = (x_1,\ldots,x_m)$, and $\nn = (x_{m+1},\ldots,x_n)$. Let $\diamond$ be the condition that $p = 1$ and $\Delta_\alpha(I^{s-|\beta|})$ and $\Delta_\beta(J^{s-|\alpha|})$ are both nonempty, and let $\diamond'$ be the condition that $p = 1$ and $\Delta_\alpha(I^{(s-|\beta|)})$ and $\Delta_\beta(J^{(s-|\alpha|)})$ are both nonempty. Then for all $p \geq 1$,

\begin{enumerate}
\item $\dim_k H_{\mm + \nn}^p(S/(I+J+\mm\nn)^s)_\gamma =$
$$\begin{cases} \dim_k H_\mm^p(A/I^{s-|\beta|})_\alpha + \dim_k H_\nn^p(B/J^{s-|\alpha|})_\beta, & \text{if not } \diamond \\
\dim_k H_\mm^p(A/I^{s-|\beta|})_\alpha + \dim_k H_\nn^p(B/J^{s-|\alpha|})_\beta + 1, & \text{if } \diamond \end{cases}$$
\item $\dim_k H_{\mm + \nn}^p(S/(I+J+\mm\nn)^{(s)})_\gamma =$
$$\begin{cases} \dim_k H_\mm^p(A/I^{(s-|\beta|)})_\alpha + \dim_k H_\nn^p(B/J^{(s-|\alpha|)})_\beta, & \text{if not } \diamond' \\
\dim_k H_\mm^p(A/I^{(s-|\beta|)})_\alpha + \dim_k H_\nn^p(B/J^{(s-|\alpha|)})_\beta, & \text{if } \diamond' \end{cases}$$
\end{enumerate}



\end{theorem}

\begin{proof} Note that $H_p(\cdot) = \tilde{H}_p(\cdot)$ for $p \geq 1$, and that $\dim_k H_p(\Delta_1 \sqcup \Delta_2) = \dim_k H_p(\Delta_1) + \dim_k H_p(\Delta_2)$ for all $p \geq 0$. Furthermore, $\dim_k H_0(\Delta) = \begin{cases} \dim_k \tilde{H}_0(\Delta) = 0, & \Delta = \emptyset, \{\emptyset\} \\ 1 + \dim_k \tilde{H}_0(\Delta), & \Delta \neq \emptyset, \{\emptyset\} \end{cases}$.

Let $K = (I + J + \mm\nn)^s$. Then by \Cref{ordinaryfiberproduct}, the nonempty faces of $\Delta(K)$ are the nonempty faces of  $\Delta_\alpha(I^{s-|\beta|}) \sqcup \Delta_\beta(J^{s-|\alpha|})$. We then have, for $p \geq 1$, that $H_p(\Delta(K)) = H_p(\Delta_\alpha(I^{s-|\beta|})) + H_p(\Delta_\beta(J^{s-|\alpha|})$.

If either $\Delta_\alpha(I^{s-|\beta|})$ or $\Delta_\beta(J^{s-|\alpha|})$ has no non-empty face, then this reduces trivially. Without loss of generality, suppose there is no $F \neq \emptyset$ in $\Delta_\beta(J^{s-|\alpha|})$. Then $G_\gamma = G_\alpha$, and by \Cref{takayama} and \Cref{ordinaryfiberproduct} that

\begin{equation} \notag
\begin{aligned}
\dim_k H_{\mm + \nn}^p(S/K)_\gamma & = & & \dim_k \tilde{H}_{p - |G_\gamma| - 1} (\Delta_\gamma((I+J+\mm\nn)^s;k) \\
& = & & \dim_k \tilde{H}_{p - |G_\alpha| - 1} (\Delta_\alpha(I^{s-|\beta|};k) \\
& = & & \dim_k H_\mm^p(\Delta_\alpha(I^{s-|\beta|};k) \\
& = & & \dim_k H_\mm^p(A/I^{s-|\beta|})_\alpha + \dim_k H_\nn^p(B/J^{s-|\alpha|})_\beta,
\end{aligned}
\end{equation}

where the last equality arises because, for $p \geq 1$, $\dim_k H_\nn^p(B/J^{s-|\alpha|})_\beta = 0$ when $\Delta_\beta(J^{s-|\alpha|})$ is the void complex or the irrelevant complex.

By \Cref{takayama} and \Cref{ordinaryfiberproduct}, we have that if neither $\Delta_\alpha(I^{s-|\beta|})$ nor $\Delta_\beta(J^{s-|\alpha|})$ is empty, then $G_\gamma = \emptyset$. Then if $p \neq |G_\gamma| + 1 = 1$,

\begin{equation} \notag
\begin{aligned}
\dim_k H_{\mm + \nn}^p(S/K)_\gamma & = & & \dim_k \tilde{H}_{p - |G_\gamma| - 1}(\Delta_\gamma((I+J+\mm\nn)^s);k) \\
& = & & \dim_k \tilde{H}_{p - 1}(\Delta_\alpha(I^{s-|\beta|}) \sqcup \Delta_\beta(J^{s-|\alpha|});k) \\
& = & & \dim_k H_{p - 1}(\Delta_\alpha(I^{s-|\beta|}) \sqcup \Delta_\beta(J^{s-|\alpha|});k) \\
& = & & \dim_k H_{p - 1}(\Delta_\alpha(I^{s-|\beta|};k)  + \dim_k H_{p - 1}(\Delta_\beta(J^{s-|\alpha|};k) \\
& = & & \dim_k \tilde{H}_{p - 1}(\Delta_\alpha(I^{s-|\beta|};k) + \dim_k \tilde{H}_{p - 1}(\Delta_\beta(J^{s-|\alpha|};k) \\
& = & & \dim_k H_\mm^p(A/I^{s-|\beta|})_\alpha + \dim_k H_\nn^p (B/J^{s-|\alpha|})_\beta
\end{aligned}
\end{equation}

If both $\Delta_\alpha(I^{s-|\beta|})$ and $\Delta_\beta(J^{s-|\alpha|})$ are nonempty and $p = |G_\gamma| + 1 = 1$, this is 

\begin{equation} \notag
\begin{aligned}
\dim_k H_{\mm + \nn}^1(S/K)_\gamma & = & & \dim_k \tilde{H}_0(\Delta_\gamma((I+J+\mm\nn)^s);k) \\
& = & & \dim_k \tilde{H}_0(\Delta_\alpha(I^{s-|\beta|}) \sqcup \Delta_\beta(J^{s-|\alpha|});k) \\
& = & & \dim_k H_0(\Delta_\alpha(I^{s-|\beta|}) \sqcup \Delta_\beta(J^{s-|\alpha|});k) - 1\\
& = & & \dim_k H_0(\Delta_\alpha(I^{s-|\beta|});k)  + \dim_k H_0(\Delta_\beta(J^{s-|\alpha|});k) - 1 \\
& = & & (\dim_k \tilde{H}_0(\Delta_\alpha(I^{s-|\beta|});k) + 1) + (\dim_k \tilde{H}_0(\Delta_\beta(J^{s-|\alpha|});k) + 1) - 1\\
& = & & \dim_k H_\mm^1(A/I^{s-|\beta|})_\alpha + \dim_k H_\nn^1(B/J^{s-|\alpha|})_\beta + 1
\end{aligned}
\end{equation}

The proof for the symbolic power case proceeds exactly the same way, using the decomposition from \Cref{symbolicfiberproduct}.
\end{proof}

Here is a lemma addressing the case when $p = 0$.

\begin{lemma} \label{pIs0}
Let $\mathcal{I} \subset S$ be a monomial ideal. Then the following are equivalent:
\begin{enumerate}
\item $\dim_k H_{\mm+\nn}^0(S/\mathcal{I})_\gamma \neq 0$ 
\item $\dim_k H_{\mm+\nn}^0(S/\mathcal{I})_\gamma = 1$
\item $\Delta_\gamma(\mathcal{I}) = \{\emptyset\}$.
\end{enumerate}
\end{lemma}

\begin{proof} For any monomial ideal $\mathcal{I} \subset S$, by \Cref{takayama},

$$\dim_k H_{\mm + \nn}^0(S/\mathcal{I})_\gamma = \dim_k \tilde{H}_{0 - |G_\gamma| - 1}(\Delta_\gamma(\mathcal{I});k).$$

For any simplicial complex $\Delta$ and for all $p \leq -1$, we have $$\tilde{H}_p(\Delta;k) = \begin{cases} 0, & p \leq -2 \text{ or } \Delta \neq \{\emptyset\} \\ k, & p = -1 \text{ and } \Delta = \{\emptyset\} \end{cases}.$$

Therefore 
\begin{equation} \notag
\begin{aligned}
\dim_k H_{\mm + \nn}^0(S/\mathcal{I})_\gamma & = & & \dim_k \tilde{H}_{0 - |G_\gamma| - 1}(\Delta_\gamma(\mathcal{I});k) \\
& = & & \begin{cases} 0, G_\gamma \neq \emptyset \text{ or } \Delta_\gamma(\mathcal{I}) \neq \{\emptyset\} \\ 1, G_\gamma = \emptyset \text{ and } \Delta_\gamma(\mathcal{I}) = \{\emptyset\} \end{cases}
\end{aligned}
\end{equation}
\end{proof}

As a consequence, together with earlier results on the symbolic power of an ideal, we are able to compute the regularity of $(I + J + \mm\nn)^{(s)}$.

\begin{corollary} \label{depthandreg} For $I \subset A$ and $J \subset B$ squarefree monomial ideals, with $\mm$ and $\nn$ the respective maximal homogeneous ideals of $A$ and $B$, we have
$$\reg(S/(I+J+\mm\nn)^{(s)}) = \max_{1 \leq t \leq s} \{\reg(A/I^{(t)}) + s - t, \reg(B/J^{(t)}) + s - t, 2s - 1\}.$$
\end{corollary}

\begin{proof}  Let 
\begin{itemize}
\item $W := \{(p,\gamma) \in \NN \times \ZZ^m \times \ZZ^{n-m} : H_{\mm+\nn}^p(S/(I+J+\mm\nn)^{(s)})_\gamma \neq 0\},$
\item $W_A := \{(p,\gamma) \in \NN_{>0} \times \ZZ^m \times \NN^{n-m} , 0 \leq |\beta| < s: H_\mm^p(A/(I^{(s-|\beta|)}))_\alpha \neq 0\}$
\item $W_B := \{(p,\gamma) \in \NN_{>0} \times \NN^m \times \ZZ^{n-m} , 0 \leq |\alpha| < s: H_\nn^p(B/(J^{(s-|\alpha|)}))_\beta \neq 0\}$
\item $W_C := \{(1,\gamma) \in \NN_{>0} \times \NN^n, |\alpha| < s, |\beta| < s, \Delta_\alpha(I^{(s-|\beta|)}) \supsetneq \{\emptyset\} \subsetneq \Delta_\beta(J^{(s-|\alpha|)})\}$
\item $W_D := \{(0,\gamma) \in \NN \times \NN^n, \Delta_\gamma((I+J+\mm\nn)^{(s)}) = \{\emptyset\}\}$
\end{itemize}

Then by \Cref{cohomologyfiberproduct}, we have the $(p,\gamma) \in W$ for which $p > 0$ are given as follows.

\begin{equation} \notag
\begin{aligned}
W \cap (\NN_{>0} \times \ZZ^n) & = & & \{(p,\gamma) \in \NN_{>0} \times \ZZ^n : H_\mm^p(A/I^{(s-|\beta|})_\alpha \neq 0 \text{ or } H_\nn^p(B/J^{(s-|\alpha|)})_\beta \neq 0\} \cup \\
& & & \{(1,\gamma) \in \NN_{>0} \times \NN^n : |\alpha| < s,|\beta| < s, \Delta_\alpha(I^{(s-|\beta|)}) \supsetneq \{\emptyset\} \subsetneq \Delta_\beta(J^{(s-|\alpha|)})\} \\
& = & & \{(p,\gamma) \in \NN \times \ZZ^m \times \NN^{n-m}, |\beta| < s : H_\mm^p(A/I^{(s-|\beta|)})_\alpha \neq 0\} \cup \\
& & & \{(p,\gamma) \in \NN \times \NN^m \times \ZZ^{n-m}, |\alpha| < s : H_\nn^p(B/J^{(s-|\alpha|)})_\beta \neq 0\} \cup W_C \\
& = & & W_A \cup W_B \cup W_C
\end{aligned}
\end{equation}

Furthermore, by \Cref{pIs0}, we have 
\begin{equation} \notag
\begin{aligned}
W \cap (\NN_{\leq 0} \times \ZZ^n) & = & & \{(0,\gamma) \in \NN \times \NN^n, \Delta_\gamma((I+J+\mm\nn)^{(s)}) = \{\emptyset\}\} \\
& = & & W_D,
\end{aligned}
\end{equation}
so that $W = W_A \cup W_B \cup W_C \cup W_D$, and $\max\{p + |\gamma| : (p,\gamma) \in W\}$ will be the maximum of $p + |\gamma|$ among $(p,\gamma)$ that appear in any of $W_A$, $W_B$, $W_C$, and $W_D$.

By \cite[Lemma 1.5]{HT16}, we have that for $(0,\gamma) \in W_D$, $|\gamma| \leq s - 1 - |G_\gamma|$, so that $\max\{p + |\gamma| : (p,\gamma) \in W_D\} \leq s - 1$.

As for $W_C$, consider $\alpha = (s-1,0,\ldots,0)$, $\beta = (s-1,0,\ldots,0)$. Then $\{1\} \in \Delta_\alpha(I^{s-|\beta|}) = \Delta_\alpha(I)$ and $\{m+1\} \in \Delta_\beta(J^{s-|\alpha|}) = \Delta_\beta(J)$, and $|\gamma|$ is maximized under the constraint that $|\alpha| < s$ and $|\beta| < s$. Therefore $\max\{p + |\gamma| : (p,\gamma) \in W_C\} = 1 + 2s - 2 = 2s - 1$. Observe then that, for $s \geq 1$, $$\max\{p + |\gamma|: (p,\gamma) \in W_C\} = 2s - 1 > s - 1 \geq \max\{p + |\gamma|: (p,\gamma) \in W_D\},$$ so that the maximum $p + |\gamma|$ for $(p,\gamma) \in W$ cannot be achieved by $(p,\gamma) \in W_D$.

For $W_A$, letting $t := s - |\beta|$, we have
\begin{equation} \notag
\begin{aligned}
\max\{p + |\gamma| : (p,\gamma) \in W_A\} & = & & \max_{1 \leq t \leq s} \{p + |\gamma| : (p,\alpha) \in \NN_{>0} \times \ZZ^m, H_\mm^p(A/(I^{(t)})_\alpha \neq 0\} \\
& = & & \max_{1 \leq t \leq s} \{s - t + \max\{p + |\alpha|: H_\mm^p(A/(I^{(t)})_\alpha \neq 0\}\} \\
& = & & \max_{1 \leq t \leq s} \{s - t + \reg(I^{(t)})\}
\end{aligned}
\end{equation}

Likewise, $\max\{p + |\gamma| : (p,\gamma) \in W_B\} = \max_{1 \leq t \leq s} \{s - t + \reg(J^{(t)})\}$.

Therefore
\begin{equation} \notag
\begin{aligned}
\reg(S/(I+J+\mm\nn)^{(s)}) & = & & \max\{p + |\gamma| \in W_A \cup W_B \cup W_C\} \\
& = & & \max\{ \max_{1 \leq t \leq s}\{s - t + \reg(I^{(t)})\}, \max_{1 \leq t \leq s}\{s - t + \reg(J^{(t)})\}, 2s - 1\} \\
& = & & \max_{1 \leq t \leq s} \{\reg(I^{(t)}) + s - t, \reg(J^{(t)}) + s - t, 2s - 1\}
\end{aligned}
\end{equation}







\end{proof}

\begin{observation} By \cite[Theorem 4.6]{GHOS18}, for $I$ the edge ideal of a simple graph (i.e., a squarefree monomial ideal generated in degree $2$), $\reg(I^{(s)}) \geq 2s - 1$. We therefore recover and extend the result of \cite{KKS19} on the regularity of symbolic powers of edge ideals of joins of graphs to symbolic powers of fiber products of any squarefree monomial ideals.

\end{observation}

\section{Degree Complex of Mixed Product} \label{sec.genmixedprod}

The degree complex decompositions of powers of sums and fiber products relied largely on the decompositions $(I + J)^s = \sum_{j=0}^s I^j J^{s-j}$ and $(I+J)^{(s)} = \sum_{j=0}^s I^{(j)} J^{(s-j)}$ and the inclusions $K^s \supseteq K^t$ and $K^{(s)} \supseteq K^{(t)}$ for $s \leq t$ for any ideal $K \subset S$. It is therefore natural to ask the more general question for the case that $s = 2$: given that $I_1 \subseteq I_2 \subseteq A$ and $J_1 \subseteq J_2 \subseteq B$, what can be said about the degree complex $I_1 J_2 + I_2 J_1$?

\begin{definition} (\cite{HT10}) Let $I_1 \subset I_2 \subset A = k[x_1,\ldots,x_m]$ and $J_1 \subseteq J_2 \subseteq B = k[x_{m+1},\ldots,x_n]$ be monomial ideals. Then the \textit{mixed product} of $I_1, I_2, J_1, J_2$ is defined to be $I_1 J_2 + I_2 J_1 \subset S = A \otimes_k B$.

\end{definition}



\begin{theorem} \label{gmp} Let $I_1 \subseteq I_2 \subseteq A$ and $J_1 \subseteq J_2 \subseteq B$ be monomial ideals, and let $H = I_1 J_2 + I_2 J_1$ be the mixed product of $I_1, I_2, J_1, J_2$. Let $\gamma = (\alpha,\beta) \in \ZZ^m \times \ZZ^{n-m}$. Then

$$\Delta_\gamma(H) = [\Delta_X \ast \Delta_\beta(J_2)] \cup [\Delta_\alpha(I_1) \ast \Delta_\beta(J_1)] \cup [\Delta_\alpha(I_2) \ast \Delta_Y]$$ 

\end{theorem}

\begin{proof} Using \Cref{adding} and \Cref{multiplying}, as well as \Cref{subsets}, we have

\begin{equation} \notag
\begin{aligned}
\Delta_\gamma(H) & = & & \Delta_\gamma(I_1 J_2) \cap \Delta_\gamma(I_2 J_1) \\
& = & & [(\Delta_X \ast \Delta_\beta(J_2)) \cup (\Delta_\alpha(I_1) \ast \Delta_Y)] \cap [(\Delta_X \ast \Delta_\beta(J_1)) \cup (\Delta_\alpha(I_2) \ast \Delta_Y)] \\
& = & & [(\Delta_X \ast \Delta_\beta(J_2)) \cap (\Delta_X \ast \Delta_\beta(J_1))] \cup [(\Delta_X \ast \Delta_\beta(J_2)) \cap (\Delta_\alpha(I_2) \ast \Delta_Y)] \cup \\
& & & [(\Delta_\alpha(I_1) \ast \Delta_Y) \cap (\Delta_X \ast \Delta_\beta(J_1))] \cup [(\Delta_\alpha(I_1) \ast \Delta_Y) \cap (\Delta_\alpha(I_2) \ast \Delta_Y)] \\
& = & & [\Delta_X \ast \Delta_\beta(J_2)] \cup [\Delta_\alpha(I_2) \ast \Delta_\beta(J_2)] \cup [\Delta_\alpha(I_1) \ast \Delta_\beta(J_1)] \cup [\Delta_\alpha(I_2) \ast \Delta_Y] \\
& = & & [\Delta_X \ast \Delta_\beta(J_2)] \cup [\Delta_\alpha(I_1) \ast \Delta_\beta(J_1)] \cup [\Delta_\alpha(I_2) \ast \Delta_Y],
\end{aligned}
\end{equation}

where the last equality comes from $[\Delta_\alpha(I_2) \ast \Delta_\beta(J_2)] \subseteq \Delta_X \ast \Delta_\beta(J_2)$.

\end{proof}


\section{Macaulay2 Code for Degree Complexes} \label{sec.m2code}

Here is a brief description of the Macaulay2 functions I have written for degree complexes, whose code is given below. The following functions are necessary to use \texttt{degreeComplex}

\texttt{range}
\begin{itemize}
\item Input: A nonnegative integer $n$
\item Output: The list $\{0,\ldots,n\}$
\end{itemize}

\texttt{negativeIndices}
\begin{itemize}
\item Input: A vector $\gamma \in \ZZ^s$
\item Output: The set $G_\gamma := \{i : \gamma_i < 0\}$.
\end{itemize}

\texttt{relevantSet}
\begin{itemize}
\item Input: A set $F \subset [n]$ and a vector $\gamma \in \ZZ^n$
\item Output: The set $[n] \setminus (F \cup G_\gamma)$
\end{itemize}

\texttt{isFace}
\begin{itemize}
\item Input: A set $F \subset [n] \setminus G_\gamma$, a squarefree monomial ideal $I$, and a vector $\gamma \in \ZZ^n$.
\item Output: ``true'' if $F \in \Delta_\gamma(I)$, and ``false'' if $F \not\in \Delta_\gamma(I)$.
\end{itemize}

\texttt{setToMonomial}
\begin{itemize}
\item Input: A set $F \subset [n]$ and a polynomial ring $S$ over $n$ variables (say $x_0,\ldots,x_{n-1}$)
\item Output: The monomial $\prod_{i \in F} x_i$
\end{itemize}

\texttt{convertToComplex}
\begin{itemize}
\item Input: A list $\mathcal{L}$ of subsets of $[n]$ and a ring $S$ with $n$ variables (say $x_0,\ldots,x_{n-1}$)
\item Output: The simplicial complex with faces $\{\{x_i : i \in F\} : F \in \mathcal{S}\}$
\end{itemize}

\texttt{degreeComplex}
\begin{itemize}
\item Input: An ideal $I \subset S$, a vector $\gamma \in \ZZ^n$, and a ring $S$ with $n$ variables.
\item Output: The degree complex $\Delta_\gamma(I)$
\end{itemize}

\begin{verbatim}
needsPackage "SimplicialComplexes"
\end{verbatim}

\begin{verbatim}
range = maximum -> 
	(ans = {}; for i from 0 to maximum-1 do ans = append(ans,i); return ans)
\end{verbatim}

\begin{verbatim}
negativeIndices = expvector -> 
	(ans = {}; for i from 0 to length(expvector)-1 do 
		if expvector_i < 0 then ans = append(ans,i); 
	return ans)
\end{verbatim}

\begin{verbatim}
relevantSet = (someset,expvector) -> 
     return toList(set(range(length(expvector))) - set(someset) 
     - set(negativeIndices(expvector)))
\end{verbatim}

\begin{verbatim}
isFace = (someset,myIdeal,expvector) -> 
     (for g in (flatten entries gens myIdeal) do(
           chex = false; 
           for i in relevantSet(someset,expvector) do( 
                 if (flatten exponents(g))_i > expvector_i then 
                     chex = true;); 
           if not chex then return false;); 
      return true)
\end{verbatim}

\begin{verbatim}
setToMonomial = (someSet,myRing) -> 
    (prod = 1; 
    for i in toList(someSet) do 
        prod = prod * (flatten entries vars myRing)#i; 
    return prod)
\end{verbatim}

\begin{verbatim}
convertToComplex = (listOfFaces,myRing) -> 
    (myList = {}; 
    for F in listOfFaces do 
        myList = append(myList,setToMonomial(F,myRing)); 
    return simplicialComplex(toList(set(myList) - set({1}))))
\end{verbatim}

\begin{verbatim}
degreeComplex = (myIdeal,expvector,myRing) -> 
    (faceList = {}; 
    for F in subsets(set(range(length(expvector))) - 
        set(negativeIndices(expvector))) do 
        if isFace(toList(F),myIdeal,expvector) then 
            faceList = append(faceList,F); 
    return convertToComplex(faceList,myRing))
\end{verbatim}


\end{document}